\documentclass[11pt]{article}
\usepackage[pagewise]{lineno}
\usepackage{amssymb}
\usepackage{latexsym,amsfonts,amssymb,amsmath,amsthm}
\usepackage{graphicx}
\usepackage{cite}
\usepackage{amsthm}
\usepackage{amsfonts}
\usepackage{xr}
\usepackage{graphicx}
\usepackage{amsmath}
\usepackage{color}
\usepackage{amssymb}
\usepackage{mdwlist}

\newtheorem{theorem}{Theorem}[section]
\newtheorem{definition}{Definition}[section]

\newtheorem{lemma}[theorem]{Lemma}
\newtheorem{rem}[theorem]{Remark}

\newtheorem{cor}[theorem]{Corollary}
\numberwithin{equation}{section}

\makeatletter % `@' now normal "letter"
\@addtoreset{equation}{section}
\makeatother  % `@' is restored as "non-letter"

\newcommand\norm[1]{\lVert#1\rVert}
\newcommand\abs[1]{\lvert#1\rvert}

\begin{document}

\title{The open dense conjecture on eventually slow oscillations of the differential equation with delayed negative feedback}

\author{Lirui Feng\thanks{School of Mathematical Sciences, University of Science and Technology of China, Hefei, Anhui, 230026, People’s Republic of China (ruilif@ustc.edu.cn, ruilif@163.com). The author is supported by NSF of China No.12101583 and No.12331006.}}
\date{}
\maketitle
\begin{abstract} In this paper, we show how to use the approach of the strongly order-preserving semiflows with respect to high-rank cones to solve the open dense conjecture on eventually slow oscillations of the differential equation with delayed negative feedback. 
\end{abstract}

\section{Introduction}

\indent In this paper, we investigate the open dense conjecture on eventually slow oscillations of the differential equation with delayed negative feedback, provided by \begin{equation}\label{E:DDE-sys-mu} x^{\prime}(t)=-\mu x(t)+f(x(t-1)),\end{equation} where $\mu\geq 0$ and $f$ is $C^1$-smoth such that $f^{\prime}<0$, $f(0)=0$ and $\sup f<+\infty$ (or $-\infty<\inf f$). The conjecture is to expect that the set of initial points of eventually slowly oscillating (see Definition \ref{SO}) solutions of Equation (\ref{E:DDE-sys-mu}) is open and dense in the phase space.  Mallet-Paret and Walther considered this conjecture in \cite{M-PWalther94} via considering rapid oscillations (see rapidly oscillating solution in Definition \ref{SO}), but the reference \cite{M-PWalther94} is not posted. The original conjecture was posed by Kaplan and Yorke in \cite{K-Y} more than fifty years ago for the special case $\mu=0$ of Equation  (\ref{E:DDE-sys-mu}). 

We show how to utilize the approach of the strongly order-preserving semiflows with respect to high-rank cones to prove that  the set of initial points of eventually slowly oscillating solutions of (\ref{E:DDE-sys-mu}) is open and dense in the phase space.  By virtue of a corollary, we also solve the open dense conjecture on eventually slow oscillations posed by Kaplan and Yorke.

The paper is organized as follows. In section 2, we present notations, definitions, and preliminary lemmas on the differential equation with delayed negative feedback (\ref{E:DDE-sys-mu}) and the strongly order-preserving semiflows with respect to high-rank cones. In section 3, we give fundamental properties of the semiflow $F_t$ generated by Equation (\ref{E:DDE-sys-mu}). In section 4, we present properties related to high-rank cones. In section 5, we solve the open dense conjecture on eventually slow oscillations of Equation  (\ref{E:DDE-sys-mu}).  

\section{Preliminary}
\subsection{The  strongly order-preserving semiflow $\Phi_t$ on a Banach space}
Now, we first give some definitions and notations on a semiflow. Let $\tilde{X}$ be a Banach space. We use $\text{Int}_{\tilde{X}},\,\,\text{Cl}_{\tilde{X}},\,\,\partial_{\tilde{X}}$ to represent the set-value mapping of taking the interior, closure and boundary of a subset in $\tilde{X}$ respectively. For convenience, we denote $\overline{A}$ the closure of $A$ in $\tilde{X}$ for any subset $A\subset \tilde{X}$ in this subsection. Let $\Phi: \mathbb{R}^+\times \tilde{X} \rightarrow \tilde{X}$ be a continuous semiflow, and we denote $\Phi_t=\Phi(t,\cdot)$ for any $t\in\mathbb{R}^+$. A subset $\mathcal{K}$ of $\tilde{X}$ is called {\it positively invariant with respect to $\Phi_t$} if $\Phi_t(\mathcal{K})\subset \mathcal{K}$ for any $t\geq0$; and it is called {\it invariant with respect to $\Phi_t$} if $\Phi_t(\mathcal{K})=\mathcal{K}$ for any $t\geq0$. Denote $O^+(x)=\{\Phi_t(x):\,t\geq 0\}$ the positive semiorbit with the initial point $x\in \tilde{X}$. The omega-limit (abbr. $\omega$-limit) set $\omega(x)$ of $O^+(x)$ with initial point $x\in\tilde{X}$ is defined by $\omega(x)=\cap_{s\geq 0}\text{Cl}_{\tilde{X}}\big(\cup_{t\geq s}O^+(\Phi_t(x))\big)$. If the positive semiorbit $O^+(x)$ is precompact in $\tilde{X}$ (i.e., its closure $\overline{O^+(x)}$ is compact in $\tilde{X}$), then $\omega(x)$ is nonempty, compact, connected, and invariant w.r.t. $\Phi$. An {\it equilibrium} is a point $x\in\tilde{X}$ for which $O^+(x)=\{x\}$ (also called a {\it trivial semiorbit}).

The set of all the equilibria of $\Phi_t$ is denoted by $E$. A nontrivial semiorbit $O^+(x)$ is said to be a {\it periodic orbit} if $\Phi_T(x)=x$ for some $T>0$, and it is called {\it $T$-periodic orbit} if there is a $T>0$ such that $\Phi_T(x)=x$ and $\Phi_t(x)\neq x$ for any $t\in(0,T)$, where $T$ is called the {\it minimal period} of $O^+(x)$.

A {\it negative semiorbit (resp. full-orbit) with initial point $x$} is sometimes viewed as a continuous function $\eta: \mathbb{R}^-=\{t\in\mathbb{R}: t\leq 0\}\mapsto \tilde{X}$ (resp. $\eta: \mathbb{R}\mapsto \tilde{X}$) such that $\eta(0)=x$, and for any $s \leq 0$ (resp. $s\in \mathbb{R}$), $\Phi_t(\eta(s))=\eta(t+s)$ holds for $0\leq t \leq -s$ (resp. $0\leq t$). Clearly, if $\eta$ is such a negative semiorbit with initial point $x$, then $\eta$ can be extended to a full-orbit with initial point $x$, denoted also by $\eta$, which is defined by $\eta(t)=\Phi_t(x),\,t>0$ on the extended domain $\mathbb{R}^+\setminus\{0\}$. On the other hand, any full-orbit with initial point $x$ restricted on $\mathbb{R}^-$ is a negative semiorbit with initial point $x$. Since $\Phi_t$ is a semiflow, a negative semiorbit with initial point $x$ may not exist, and it may not be unique even if exists. Hereafter, we denote by $O_b^-(x)=\{\eta(t):\,t\leq 0\}$ \big(resp. $O_b(x)=\{\eta(t):\,t\in \mathbb{R}\}$\big) a negative semiorbit (resp. full-orbit) with initial point $x$ in order to emphasize its initial point in the phase space $\tilde{X}$ and to coincide with the feature of a positive semiorbit being a subset of $\tilde{X}$. In particular, we write them as $O^-(x)$ (resp. $O(x)$) if such a negative semiorbit (resp. full-orbit) with initial point $x$ is unique. Clearly, for any $x\in \mathcal{K}$, there exists a negative semiorbit with initial point $x$, provided that $\mathcal{K}$ is invariant w.r.t. $\Phi_t$. Given a negative semiorbit $O_b^-(x)$ with initial point $x$, if $O_b^-(x)$ is precompact, then the $\alpha$-limit set $\alpha_b(x)$ of $O_b^-(x)$ is defined by $\alpha_b(x)=\cap_{s\geq 0}\text{Cl}_{\tilde{X}}\big(\cup_{t\geq s}\{ \eta_{-\tau}(x):\,\tau\geq t\}\big)$. We write $\alpha(x)$ as the $\alpha$-limit set of $O^-(x)$ if there is a unique negative semiorbit with initial point $x$. For a given invariant set $\mathcal{K}$ w.r.t. $\Phi_t$, $\Phi_t$ is said {\it admits a flow extention on $\mathcal{K}$} if there is a flow $\tilde{\Phi}_t,\,t\in\mathbb{R}$ on $\mathcal{K}$ such that $\tilde{\Phi}_t(x)=\Phi_t(x)$ for any $x\in\mathcal{K}$ and $t\geq 0$. For convenience, we denote $\Phi_{-t}(x)$ the point such that $\Phi_t (\Phi_{-t}(x))=x, \,t\geq 0$, and hence $O^-(x)=\{\Phi_{-t}(x):\,\,t\geq 0\}$ for any $x\in \mathcal{K}$ if $\Phi_t$ admits a flow extention on $\mathcal{K}$.

A closed subset $C$ of $\tilde{X}$ is called {\it a cone of rank $k$} (abbr. $k$-cone) if the following hold: (i) $\mathbb{R}\cdot C\subset C$; (ii) $\max\{\text{dim}W: W\subset C\,\,\text{linear subspace}\}=k$. A $k$-cone $C$ is called {\it solid} if its interior ${\rm \text{Int}}_{\tilde{X}}C\neq \emptyset$ and it is called {\it $k$-solid} if there exists a $k$-dimensional linear subspace $W$ such that $W\setminus \{0\}\subset {\rm \text{Int}}_{\tilde{X}}C$. Given a $k$-cone $C$, we say that $C$ is complemented if there exists a k-codimensional linear subspace $H^c\subset \tilde{X}$ such that $H^c \cap C=\{0\}$.

Two points $x,y\in\tilde{X}$ are called {\it ordered {\rm(}with respect to $C${\rm)}}, denoted by $x\thicksim y$, if $x-y\in C$; otherwise, they are called {\it unordered {\rm(}with respect to $C${\rm)}}, denoted by $x\rightharpoondown y$. For simplicity, we omit ``with respect to $C$'' throughout this subsection for concepts related to ``order'' which is with respect to $C$, and use the same symbols like ``$\thicksim$, $\rightharpoondown$'' to denote relationships on orders w.r.t. different special $k$-cones throughout this paper. Two points $x,y\in\tilde{X}$ are called {\it strongly ordered}, denoted by $x\approx y$, if $x-y\in {\rm\text{Int}}_{\tilde{X}}C$. For sets $A,\,B\subset \tilde{X}$, we write $A\thicksim B$ if $x-y\in C$ for any $x\in A$ and $y\in B$.  A subset $W\subset \tilde{X}$ is called {\it ordered} if $x \thicksim y$ for any $x,y\in W$. $W$ is called {\it unordered {\rm(}resp. strongly ordered{\rm)}} if there are at least two distinct points in $W$ and $x\rightharpoondown y$ ({\it resp. $x \thickapprox y$}) for any two distinct points $x,y\in W$.

$\Phi_t$ is called {\it monotone with respect to $C$} if $\Phi_t(x) \thicksim \Phi_t(y)$ for any $t>0$ whenever $x\thicksim y$; and it is called {\it strongly monotone with respect to $C$} if $\Phi_t(x) \approx \Phi_t(y)$ for any $t>0$ whenever $x\thicksim y$ with $x\neq y$. 

\begin{definition}
A semiflow $\Phi_t$ is called {\it strongly order-perserving {\rm(}abbr. SOP{\rm)}} with respect to $C$, if it is monotone with respect to $C$, and whenever $x\thicksim y$ with $x\neq y$, there are a $T>0$ and open neighborhoods $\mathcal{U}$ of $x$, $\mathcal{V}$ of $y$ such that $\Phi_t(\mathcal{U}) \thicksim\Phi_t(\mathcal{V})$ for all $t\geq T$. 
\end{definition}

A nontrivial positive semiorbit $O^+(x)$ is called {\it pseudo-ordered} if there exist two distinct points $\Phi_{t_1}(x),\, \Phi_{t_2}(x)$ in $O^+(x)$ such that $\Phi_{t_1}(x)\thicksim \Phi_{t_2}(x)$; otherwise, it is called {\it unordered}.

\begin{lemma}\label{17-A} Assume that $\Phi_t$ on $\tilde{X}$ is a semiflow SOP with respect to a solid $k$-cone $C$, which admits a flow extension on each nonempty omega-limit set of its positive semiorbits. If $O^+(x)$ is a precompact pseudo-ordered semiorbit, then the closure of any full-orbit in $\omega(x)$ is ordered. 
\end{lemma}

\begin{proof} See \cite[Theorem A]{F-W-W}. We piont out that the proof of \cite[Theorem A]{F-W-W} is also applicable to this lemma even if the assumption $\Phi_t$ being strongly monotone w.r.t. a solid $k$-cone $C$ in \cite{F-W-W} is here weakened as the one $\Phi_t$ being SOP w.r.t. a solid $k$-cone $C$.
\end{proof}

Hereafter, we assume that assumptions of Lemma \ref{17-A} hold in this subsection.

\begin{lemma}\label{17-B} Let $O^+(x)$ be a precompact pseudo-ordered semiorbit. Then one of the following three alternatives must hold: 
\item[(i)] either, $\omega(x)$ is ordered; or,
\item[(ii)] $\omega(x)\subset E$ is unordered; or otherwise,
\item[(iii)] $\omega(x)$ possesses an ordered homoclinic property, i.e., there is an ordered and invariant subset $\tilde{B}\subset \omega(x)$ w.r.t. $\Phi_t$ such that $\tilde{B}\thicksim \omega(x)$, and for any $p\in\omega(x)\setminus \tilde{B}$, we have that $\alpha(p)\cup \omega(p)\subset \tilde{B}$, $\alpha(p)\subset E$, and there is a $\tilde{p}\in \omega(x)\setminus \tilde{B}$ satisfying $\tilde{p}\rightharpoondown p$.

In particular, if $\omega(x)\cap E=\emptyset$, then $\omega(x)$ is ordered. Moreover, if $C$ is complemented, then $\Phi_t$ on $\omega(x)$ is topologically conjugate to a flow $\Psi_t$ on a compact set invariant w.r.t $\Psi_t$ in $\mathbb{R}^k$.
\end{lemma}

\begin{proof} 
See \cite[Theorem B, Proposition 4.4-4.5, 4.8 and Lemma 4.7]{F-W-W} and proofs therein. For the sake of completeness, we here give some details. By virtue of Lemma \ref{17-A}, we assume the following to proceed the analysis without loss of generality: 

$${\bf(F)}\,\,\,\,\,\,\,\,\overline{O(a)}\subsetneq\omega(x)\,\,\,\text{for any}\,\,\,a\in \omega(x).\,\,\,\,\,\,\,\,\,\,\,$$ Here, $\overline{O(a)}$, the closure of full-orbit $O(a)$ in $\omega(x)$, can be classified into two types:

\item[(P1)] $\overline{O(a)}\thicksim \omega(x)$;

\item[(P2)] there is some $z\in \omega(x)\setminus \overline{O(a)}$ such that $z\rightharpoondown y$ for some $y\in O(a)$.

\noindent Define

$$\begin{aligned}&B=\{\overline{O(b)}\subsetneq \omega(x): \overline{O(b)} \,\,\text{satisfies}\,\,\text{(P1)}\},\,\,\,\tilde{B}=\bigcup\limits_{\overline{O(b)}\in B}\overline{O(b)},\,\,\,\,\text{and}\\
&A=\{\overline{O(a)}\subsetneq \omega(x): \overline{O(a)}\,\,\text{satisfies}\,\,\text{(P2)}\},\,\,\tilde{A}=\bigcup\limits_{\overline{O(a)}\in A}\overline{O(a)}.\end{aligned}$$

\noindent Clearly, $A\cap B=\emptyset$, $\tilde{A}\cup\tilde{B}=\omega(x)$, $\tilde{B}\thicksim\omega(x)$ (hence, $\tilde{B}$ is ordered) and for any $p\in \omega(x)\setminus \tilde{B}$, there is a $\tilde{p}\in\omega(x)\setminus \tilde{B}$ such that $\tilde{p}\rightharpoondown p$. 

We point out that all the arguments in the proofs of \cite[Theorem B and Proposition 4.4-4.5, 4.8 and Lemma 4.7]{F-W-W} with their used priori results are also applicable to this lemma even if the assumption $\Phi_t$ being strongly monotone w.r.t. a solid $k$-cone $C$ in \cite{F-W-W} is here weakened as the one $\Phi_t$ being SOP w.r.t. a solid $k$-cone $C$. 

Under the additional technical assumption ({\bf F}), the alternative (i) and (iii) are deduced for the case $A=\emptyset$ and the case $A\neq\emptyset,\,B\neq\emptyset$ respectively; and more, the alternative (ii) is deduced for the case $B=\emptyset$. \cite[Proposition 4.5]{F-W-W} is devoted to the alternative (iii), and \cite[Proposition 4.8 with Lemma 4.7]{F-W-W} is devoted to the alternative (ii).
\end{proof}

\begin{lemma}\label{17-C} Let $O^+(x)$ be a precompact pseudo-ordered semiorbit in $\tilde{X}$ such that $\omega(x)$ contains no equilibrium.
Then $\omega(x)$ consists of a nontrivial periodic orbit, that is ordered.
\end{lemma}

\begin{proof} See \cite[Theorem C]{F-W-W}. We should point out that by the flow extension assumption of $\Phi_t$ on $\omega(x)$, this result can also be established for continous semiflows SOP w.r.t. high-rank solid cones.
\end{proof}

\begin{lemma}\label{17-D} Let $O^+(x)$ be a precompact unordered semiorbit. Then $\omega(x)$ is either an equilibrium or an unordered set.
\end{lemma}
\begin{proof} See \cite[Theorem D]{F-W-W}. We point out that the proof of \cite[Theorem D]{F-W-W} remains valid even if the assumption $\Phi_t$ being strongly monotone w.r.t. a solid $k$-cone $C$ in \cite{F-W-W} is here weakened as the one $\Phi_t$ being SOP w.r.t. a solid $k$-cone $C$. 
\end{proof}

\begin{lemma}\label{17-L4.1} Assume that $K_1,\,K_2\subset \tilde{X}$ are compact sets such that $K_1\cap K_2=\emptyset$ and $K_1\thicksim K_2$. Then, there are open sets $U,\,V$ and a $T>0$ such that $K_1\subset U$,$K_2\subset V$ and $\Phi_t(U)\thicksim \Phi_t(V)$ for any $t>T$.
\end{lemma}
\begin{proof} See \cite[Lemma 4.1]{F-W-W}, whose proof is also applicable to this lemma.
\end{proof}

\begin{lemma}\label{17-L4.3} Assume that there are $a_1,\,a_2\in \omega(x)$ such that $a_1\neq a_2$, $a_1\thicksim a_2$, and $O^+(x)$ is a precompact positive semiorbit of $\Phi_t$. If there is a sequence 
$\tau_n\rightarrow +\infty$ such that $\Phi_{\tau_n}(a_1)\rightarrow c$ and $\Phi_{\tau_n}(a_2)\rightarrow c$ as $n\rightarrow +\infty$, then either $c$ is an equilibrium or $O(c)$ is nontrivial and ordered.
\end{lemma}

\begin{proof} See \cite[Lemma 4.3]{F-W-W}, whose proof is also applicable to this lemma.
\end{proof}

\begin{lemma}\label{order-rela-betw-ome} Let $x,y\in \tilde{X}$ be given such that $x\thicksim y$ with $x\neq y$ and $O^+(x)$, $O^+(y)$ are precompact. Then the following results hold:
\item[(i)] If $\omega(x)\cap \omega(y)=\emptyset$, then $\omega(x)\thicksim \omega(y)$;
\item[(ii)] If $\omega(x)\cap \omega(y)\neq\emptyset$ and $\omega(x)\setminus \omega(y)\neq\emptyset$, then there is a full-orbit $O(z)\subset \omega(x)\cap \omega(y)$ such that $\omega(x)\thicksim \overline{O(z)}$. Consequently, $O^+(x)$ is a pseudo-ordered semiorbit;
\item[(iii)] If $\omega(x)=\omega(y)$, then $\omega(x)$ either consists of equilliria, or contains a nontrivial full-orbit such that the full-orbit being an ordered set.
\end{lemma}

\begin{proof} {\rm The proof of (i) and (ii)}:  Note that assumptions in (i) and (ii) are two different subcases under the assumption $\omega(x)\setminus \omega(y)\neq\emptyset$. We firstly deduce conclusions from the assumption $\omega(x)\setminus \omega(y)\neq\emptyset$. 

Given any $z_1\in \omega(x)\setminus \omega(y)$, there is an increasing sequence $\{t_n\}\subset \mathbb{R}^+$ such that $\lim\limits_{n\rightarrow +\infty}t_n=+\infty$ and $\lim\limits_{n\rightarrow +\infty}\Phi_{t_n}(x)=z_1$.  One can suppose that $\lim\limits_{n\rightarrow +\infty}\Phi_{t_n}(y)=z_2\in\omega(y)$ without loss of generality. Clearly, $z_1\notin \overline{O(z_2)}$ and $z_1\thicksim z_2$ because of $z_1\in \omega(x)\setminus \omega(y)$ and $\Phi_t$ being SOP. Let 

$$\Gamma(z_1,\,z_2)=\sup\{t\geq 0:\,z_1\thicksim\Phi_s(z_2)\,\text{for all}\,s\in[-t,t] \}$$ and $h(t)=\Gamma(\Phi_t(x),\Phi_t(y))$ for any $t\geq 0$. By utilizing the SOP property of $\Phi_t$ again, we have $\Gamma(\Phi_{t_n}(x), \Phi_{t_n}(y))\leq \Gamma(z_1,\,z_2)$ for any $n\geq 0$, and more,  $\Gamma(\Phi_{t}(x), \Phi_{t}(y))\leq \Gamma(z_1,\,z_2)$ for any $t>0$.

We assert that $\Gamma(z_1, z_2)=+\infty$. We prove it by using contradiction. Suppose that $\Gamma(z_1, z_2)<+\infty$. Let $K=\{\Phi_s(z_2): s\in[-\Gamma(z_1,z_2),\,\Gamma(z_1,z_2)]\}$. Clearly, $K\subset \omega(y)$ is compact because of the compactness of $\overline{O^+(y)}$. Then, $z_1\notin K$. By virtue of Lemma \ref{17-L4.1}, there exist open neighborhoods $U$ of $z_1$, $V$ of $K$ and a time $T>0$ such that $U\cap V=\emptyset$ and $\Phi_t(U)\thicksim \Phi_t(V)$ for all $t\geq T$. Thus, there exist $N\in\mathbb{N}^+$ and $\epsilon>0$ such that $\Phi_{t_N}(x)\in U$ and $\Phi_{s+l+t_N}(y)\in V$ for any $s\in [-\Gamma(z_1,z_2),\,\Gamma(z_1,z_2)]$ and $l\in[-\epsilon,\epsilon]$. Consequently, $\Phi_{t_N+T}(x)\thicksim \Phi_{t_N+T+s}(y)$ for any $s\in [-\Gamma(z_1,z_2)-\epsilon, \Gamma(z_1,z_2)+\epsilon]$. It follows that $\Gamma(\Phi_{t_N+T}(x), \Phi_{t_N+T}(y))\geq \Gamma(z_1,\,z_2)+\epsilon$, a contradiction. Therefore, $\Gamma(z_1, z_2)=+\infty$, which implies that $z_1\thicksim \overline{O(z_2)}$.

Now, we claim that $\omega(x)\thicksim \overline{O(z_2)}$.  By utilizing Lemma \ref{17-L4.1} again, we can find open neighborhoods $\tilde{U}$ of $z_1$, $\tilde{V}$ of $\overline{O(z_2)}$ and $\tilde{T}>0$ such that $\tilde{U}\cap \tilde{V}=\emptyset$ and $\Phi_t(\tilde{U})\thicksim \Phi_t(\tilde{V})$ for all $t\geq \tilde{T}$. It then follows from the invariance of $\overline{O(z_2)}$ w.r.t. $\Phi_t$ and the SOP property of $\Phi_t$ that $\Phi_t(x)\thicksim \overline{O(z_2)}$ for all sufficiently large $t$. Hence, $\omega(x)\thicksim \overline{O(z_2)}$.

If $\omega(x)\cap\omega(y)=\emptyset$, then $\omega(x)\cap \overline{O(z_2)}=\emptyset$. By the similiar arguments in the last paragraph, we have $\omega(x)\thicksim \omega(y)$. Thus, we complete the proof of (i).

If $\omega(x)\cap\omega(y)\neq\emptyset$, then one of the following two alternatives holds: (1) $\overline{O(z_2)}\subset \omega(x)$; (2) $\exists\,z_3\in \overline{O(z_2)}\setminus\omega(x)$. Under the assumption (1), the conclusion of (ii) is directly implied by $\omega(x)\thicksim \overline{O(z_2)}$. Under the assumption (2), it is clear that $\omega(x)\thicksim z_3$ by utilizing $\omega(x)\thicksim \overline{O(z_2)}$ again. By repeating the arguments for proving $\omega(x)\thicksim \overline{O(z_2)}$ with the conditions $z_1\thicksim\overline{O(z_2)}$ and $z_1\in \omega(x)\setminus\overline{O(z_2)}$, we get $\omega(x)\thicksim\omega(y)$. Consequently, $\omega(x)\thicksim\omega(x)\cap\omega(y)$. Recall that $\Phi_t$ admits a flow extension on each nonempty $\omega$-limit set. Take any $\tilde{z}\in\omega(x)\cap\omega(y)$. The invariance of $\omega(x)$ and $\omega(y)$ w.r.t. $\Phi_t$ implies $\omega(x)\thicksim\omega(\tilde{z})\subset \omega(x)\cap\omega(y)$. Take $z\in\omega(\tilde{z})$. Then $O(z)\subset \omega(x)\cap\omega(y)$ is the desired full-orbit. Therefore, we have completed the proof of (ii).

{\rm (iii)} For the case $\omega(x)$ consisting of equlibria, we have done. So, we only need to consider the case there is a $z_1\in \omega(x)\setminus E$. Clearly, there is an increasing sequence $\{\tau_n\}\subset \mathbb{R}^+$ such that $\lim\limits_{n\rightarrow +\infty}\tau_n=+\infty$, $\lim\limits_{n\rightarrow +\infty}\Phi_{\tau_n}(x)=z_1$ and $\lim\limits_{n\rightarrow +\infty}\Phi_{\tau_n}(y)=z_2$ for some $z_2\in \omega(y)=\omega(x)$. 

If $z_1\neq z_2$, then $z_1\thicksim z_2$, and hence, the SOP property of $\Phi_t$ w.r.t. $C$ and the assumption $\Phi_t$ admitting a flow extension on $\omega(x)$ imply that $O^+(x)$ is a pseudo-ordered semiorbit. By Lemma \ref{17-A}, each full-orbit in $\omega(x)$ is an ordered set and it then follows that $\overline{O(z_1)}$ is nontrivial and ordered. If $z_1= z_2$, it then follows from Lemma \ref{17-L4.3} that $O(z_1)$ is nontrivial and ordered. Thus, we have proved (iii).

Therefore, we have completed the proof. 
\end{proof}

\subsection{The delayed equation and its generated semiflow $F_t$}
Let $X=C[-1,0]$ be equipped with the maximal norm $\norm{\cdot}$, and $X^1=C^{1}[-1,0]$ be equipped with the $C^1$-norm $\norm{\cdot}_{X^1}$. For any $\psi\in X^{1}$, $\psi^{\prime}(-1)$ and $\psi^{\prime}(0)$ represent the right and left derivatives of $\psi$ respectively. Denoted by $x^{\phi}$ the unique solution of Equation (\ref{E:DDE-sys-mu}) with the initial point $\phi\in X$. Define the point $x^{\phi}_t\in X\,(t\geq 0)$ by $x^{\phi}_{t}(\theta)=x^{\phi}(t+\theta),\,\theta\in[-1,0]$ (Sometimes, it is also denoted by $x_t$ when we need not to point out or tag the initial point of the solution $x$ of (\ref{E:DDE-sys-mu})). A point $\tau\in \mathbb{R}$ is called a {\it zero} of a solution $x$ of (\ref{E:DDE-sys-mu}) if $x(\tau)=0$.

\begin{lemma}\label{global-solution} There exists a constant $b(\mu,f)\in(0,+\infty)$ such that $\lim\sup\limits_{t\rightarrow +\infty}\norm{x^{\phi}_t}\leq b(\mu,f)$ for all $\phi\in X$.
\end{lemma}

\begin{proof} Clearly, for any $\phi\in X$, the segement $x^{\phi}_t\,(t\geq 0)$ of the solution $x^{\phi}$ of Equation $(\ref{E:DDE-sys-mu})$ is contained in $X$. Each $x^{\phi}$ has one of the following two types:  \item[Type (a).]  All zeros in $[-1,+\infty)$ of $x^{\phi}$ are contained in a bounded interval; \item[Type (b).] The set of zeros in $[-1,+\infty)$ of $x^{\phi}$ is unbounded.

{\it The proof for Type {\rm(}a{\rm)}}: Clearly, there is a time $T_0>1$ such that $x^{\phi}$ has no zero in $[T_0-1,\infty)$ and hence, $\text{sign}(x^{\phi}(t))=\text{sign}(x^{\phi}(T_0))$ for any $t\geq T_0-1$. Without loss of generality, suppose $x^{\phi}(t)>0$ for any $t\geq T_0-1$. It follows from Equation ($\ref{E:DDE-sys-mu}$) that $x^{\prime}(t)<0$ for any $t\geq T_0$. So, there is a $c\geq 0$ such that $\lim\limits_{t\rightarrow \infty}x^{\phi}(t)=c$ and $\lim\limits_{t\rightarrow \infty}(x^{\phi})^{\prime}(t)=-\mu c+f(c)$. If $c>0$, then $-\mu c+f(c)<0$ and hence $\lim\limits_{t\rightarrow\infty} x^{\phi}(t)=-\infty$, a contradiction. Therefore, $c=0$.

{\it The proof for Type {\rm(}b{\rm)}}: We shall show that there exists a constant $b(\mu,f)\in(0,\infty)$ such that 

$$\lim\sup\limits_{t\rightarrow \infty}\norm{x^{\phi}_t}\leq b(\mu,f),$$ where $b(\mu,f)$ is independent on $\phi$. We prove it by estimating the upper and lower bounds of $x^{\phi}$ on a interval $J\subset [-1,+\infty)$ such that $\sup J=+\infty$ under one of the two different additional assumptions: 1) $\sup f<+\infty$, 2) $-\infty<\inf f$.

1) Proof with the assumption $\sup f<+\infty$. Since the set of zeros in $[-1,+\infty)$ of $x^{\phi}$ is unbounded, there must be a $t$ such that $t\geq n$ and

\begin{equation}\label{extrema}0=(x^{\phi})^{\prime}(t)=-\mu x^{\phi}(t)+f(x^{\phi}(t-1))\end{equation} for any $n\in\mathbb{N}^+$. Together with (\ref{E:DDE-sys-mu}), one of the following two alternatives holds: either,
\item[(a1)] there is a $t_0>-1$ such that $x^{\phi}(s)=0$ for all $s\geq t_0$; or otherwise, 

\item[(b1)] for any $T>1$, there is a local maximum $\overline{t}>T$ and a local minimum $\tilde{t}>T$ such that $x^{\phi}(\overline{t})>0$ and $x^{\phi}(\tilde{t})<0$.   

\noindent It suffices to prove the boundedness of $x^{\phi}$ for the case (b1). Given a $T=T_0>1$. Recall that $f^{\prime}<0$. Clearly, $\overline{t}$ and $\tilde{t}$ satisfy (\ref{extrema}). Then, $x^{\phi}(\overline{t})>0$ implies $x^{\phi}(\overline{t}-1)<0$ for any local maximium $\overline{t}>T_0$ with $x^{\phi}(\overline{t})>0$. It follows that there exists a largest zero $\tau\in(\overline{t}-1,\overline{t})$ of $x^{\phi}$ such that 

\begin{equation}\label{upper-bound}0<x^{\phi}(\overline{t})=\int_{\tau}^{\overline{t}} (x^{\phi})^{\prime}(s)ds\leq \int_{\tau}^{\overline{t}} f(x^{\phi}(s-1))ds \leq \sup f<+\infty.\end{equation} Together with (\ref{extrema}), $x^{\phi}(\tilde{t})<0$ implies $x^{\phi}(\tilde{t}-1)>0$ for any local minimium $\tilde{t}>\overline{t}_{T_0}+2$ with $x(\tilde{t})<0$, where $\overline{t}_{T_0}$ is the smallest local maximium belonging to the interval $(T_0,\,+\infty)$. Consequently, there exists a largest zero $\tilde{\tau}\in (\tilde{t}-1,\tilde{t})$ such that \begin{equation}\begin{aligned} &x(\tilde{t})=\int_{\tilde{\tau}}^{\tilde{t}}x^{\prime}(s)ds=\int_{\tilde{\tau}}^{\tilde{t}}-\mu x(s)+f(x(s-1))ds\\ 
&>\int_{\tilde{\tau}}^{\tilde{t}} f(x(s-1))ds\overset{(\ref{upper-bound})+(f^{\prime}<0)}{\geq}\int_{\tilde{\tau}}^{\tilde{t}}f(\sup f) ds>f(\sup f).\end{aligned}\end{equation}  Thus, one can take $b(\mu,f)=\sup f-f(\sup f)$ when the condition $\sup f<+\infty$ holds.

2) Proof with the assumption $-\infty<\inf f$. The arguments are similiar to the case 1). Equation (\ref{extrema}) still holds for any local extrema $t\geq 1$.  Moreover, one of the following two alternatives holds: either,
\item[(a2)] there is a $t_0>-1$ such that $x^{\phi}(s)=0$ for all $s>t_0$; or otherwise, 

\item[(b2)] for any $T>1$, there is a local minimum $\tilde{t}>T$ and a local maximum $\overline{t}>T$ such that $x^{\phi}(\tilde{t})<0$ and $x^{\phi}(\overline{t})>0$. 

It suffices to prove the boundedness of $x^{\phi}$ for the case $(b2)$. Given a $T=T_0>1$. We can obtain $x(\tilde{t})\geq\inf f$ for any local minimium $\tilde{t}>T_0$, and more, $x(\overline{t})\leq f(\inf f)$ for any local maximium $\overline{t}>\tilde{t}_{T_0}+2$, where $\tilde{t}_{T_0}$ is the smallest local minimium belonging to the interval $(T_0,+\infty)$. Thus, one can take $b(\mu,f)=f(\inf f)-\inf f$ when the condition $-\infty<\inf f$ holds.
\vskip 3mm
Therefore, we have completed the proof.
\end{proof}

\begin{rem} By virtue of Lemma \ref{global-solution}, it is clear that for any initial point $\phi\in X$, the solution $x^{\phi}$ is a function with its domain containing $[-1,+\infty)$.
\end{rem}

\begin{definition}\label{SO}{\rm A solution $x$ of (\ref{E:DDE-sys-mu}) is called {\it slowly oscillating {\rm(}with respect to the delay time 1{\rm)}} if either there is no two different zeros of $x$ in $[-1,+\infty)$, or $\abs{\tau_2-\tau_1}>1$ for each pair of its different zeros ($\tau_1,\tau_2)\in [-1,+\infty)\times [-1,+\infty)$ when $x$ has at least two different zeros in $[-1,+\infty)$. Moreover, $x$ is called {\it eventually slowly oscillating {\rm(}with respect to the delay time 1{\rm)}} if there is an interval $J\subset [-1,+\infty)$ with $\sup J=+\infty$ such that either there is no two different zeros of $x$ in $J$, or $\abs{\tau_2-\tau_1}>1$ for each pair of its different zeros ($\tau_1,\tau_2)\in J\times J$ when $x$ has at least two different zeros in $J$; otherwise, $x$ is called {\it rapidly oscillating}. }
\end{definition}

Let $\mathcal{S}$ be the set of all nonzero functions in $X$ with at most one sign change:

$$\mathcal{S}=\{\phi\in X\setminus\{0\}:\exists\,\,j\in \mathbb{N},\tau\in[-1,0]\,\,\text{s.t.}\,\,(-1)^{j}\phi\leq 0 \,\,\text{on}\,\,[-1,\tau],\,\,\text{and}\,\, (-1)^{j}\phi\geq 0 \,\, \text{on}\,\,[\tau,0]\}.$$ Clearly, the closure of $\mathcal{S}$ in $X$ is $\text{Cl}_{X}\mathcal{S}=\mathcal{S}\cup \{0\}$. Denote $\mathcal{S}^{1}=\mathcal{S}\cap X^1$. Moreover, the segment $x_t$ of a slowly oscillating solution $x$ belongs to the set $\mathcal{S}$. Let $\mathcal{D}_{ES}$ consist of all initial points in $X$ such that the corresponding solution of (\ref{E:DDE-sys-mu}) is eventually slowly oscillating.

\begin{rem}It follows from Lemma \ref{Int-bounary-S1} and \ref{monotonicity}(i) that the following are equivalent for a solution $x$ of (\ref{E:DDE-sys-mu}):

{\rm (i)} $x$ is rapidly oscillating;

{\rm (ii)} For any $n\in \mathbb{N}^+$, there is a zero $\tau$ of $x$ in its domain such that $\tau>n$; and more, one can find a zero $\tilde{\tau}$ of $x$ in its domain such that $0<\tilde{\tau}-\hat{\tau}<1$ for each zero $\hat{\tau}$ of $x$ in its domain.
\end{rem}
\vskip 3mm

By Lemma \ref{global-solution}, Equation (\ref{E:DDE-sys-mu}) generates the semiflow $F: \mathbb{R}^+\times X\rightarrow X$, which is defined by 

$$F(t,\phi)=x_{t}^{\phi}, (t,\phi)\in \mathbb{R}^+\times X.$$ For simplicity, denote $F(t,\cdot)=F_t$ for any $t\geq0$. Together with the continuity of solutions of Equation (\ref{E:DDE-sys-mu}) and \cite[Theorem 2.2-2.3 in Chapter 2]{Hale-Lun}, one can obtain the continuity of $F$. The following integral equation is deduced from (\ref{E:DDE-sys-mu}): 

\begin{equation}\label{int-eq}x^{\phi}(t)=e^{-\mu t}\phi(0)+\int_{0}^te^{-\mu(t-s)}f(x^{\phi}(s-1))ds\end{equation} for any $(t,\phi)\in \mathbb{R}^+\times X$. Then, 

\begin{equation}\label{reprentation-semiflow}  F_t(\phi)(0)=F_1(\phi)(t-1)=e^{-\mu t}\phi(0)+\int_{0}^t e^{-\mu(t-s)}f(\phi(s-1))ds \end{equation} for any $(t,\phi)\in [0,1]\times X$. The linear variational equation along the unique solution $x^{\phi}$ of (\ref{E:DDE-sys-mu}) with the initial point $\phi\in X$ is 

\begin{equation}\label{E:LVE} v^{\prime}=-\mu v(t)+f^{\prime}(x^{\phi}(t-1))v(t-1) .\end{equation} Let $v^{\xi}_{\phi}$ be the solution of (\ref{E:LVE}) with the intial data $\xi \in X$. Define $(v^{\xi}_{\phi})_t(\cdot)$ by $(v^{\xi}_{\phi})_t(\theta)=v^{\xi}_{\phi}(t+\theta)$ with $\theta\in[-1,0]$. Recall that $f$ is $C^1$-smooth. Then, $f^{\prime}$ is uniformly continuous on each bounded interval $J\subset\mathbb{R}$. It then follows from (\ref{int-eq})-(\ref{reprentation-semiflow}) that for any $(t,\phi)\in [0,1]\times X$, the $\phi$-derivative of $F_t$, denoted by $D_{\phi} F_t$, exists and 

\begin{equation}\label{DFt}D_{\phi}F_t\xi=(v^{\xi}_{\phi})_t\,\,\text{for any}\,\,\xi\in X.\end{equation} Moreover, (\ref{DFt}) holds for any $(t,\phi,\xi)\in\mathbb{R}^+\times X\times X$, and there exists the following chain-rule: 

\begin{equation}\label{Chain-DFt} D_{\phi}F_{t_1+t_2}=D_{F_{t_1}(\phi)}F_{t_2}\circ D_{\phi}F_{t_1}\,\,\text{for any}\,\,(t_1, t_2, \phi)\in \mathbb{R}^+\times\mathbb{R}^+\times X.\end{equation}

\section{Fundamental properties of the semiflow $F_t$}

\begin{lemma}\label{Le: embeding} $F_t(X)\subset X^1$ for any $t\geq 1$.
\end{lemma}

\begin{proof} It is directly implied by Equation (\ref{E:DDE-sys-mu}).
\end{proof}

\begin{lemma}\label{Le: smoothness}  $F_t$, for any fixed $t\in\mathbb{R}^+$, is $C^1$-smooth in $X$.
\end{lemma}

\begin{proof} It suffices to prove that $D_{(\cdot)}F_t: X \mapsto L(X)$ is continuous for any fixed $t\geq 0$, where $L(X)$ is the Banach space consisting of all bounded linear operators on $X$. It then follows the $C^1$-smoothness of the map $F_t$ for any fixed $t\in\mathbb{R}^+$. Recall that  $v^{\xi}_{\phi}$ is the unique solution of $(\ref{E:LVE})$ with the initial data $\xi\in X$ for any $\phi\in X$. It follows from (\ref{E:LVE}) and (\ref{DFt}) that 

\begin{equation}\label{derivative-evolution}(D_{\phi}F_t\xi)(0)=e^{-\mu t}\xi(0)+\int_{0}^t e^{-\mu(t-s)}f^{\prime}(x^{\phi}(s-1))v^{\xi}_{\phi}(s-1)ds
\end{equation} for any $(t,\phi,\xi)\in\mathbb{R}^+\times X\times X$. Denoted by $S=\{\xi\in X: \norm{\xi}=1\}$. Then, for any $t\in [0,1]$ and $\tilde{\phi},\,\phi\in X$, 

\begin{equation}\begin{aligned}& \,\,\,\,\,\,\,\,\norm{D_{\tilde{\phi}}F_t-D_{\phi}F_t}_{L(X)}=\sup\limits_{\xi\in S}\{\norm{D_{\tilde{\phi}}F_t\xi-D_{\phi}F_t\xi}\}\\&=\sup\limits_{\xi\in S}\{\max\limits_{\mathop{0\leq t+\theta\leq 1,\,\theta\in[-1,0]}}\{\abs{v^{\xi}_{\tilde{\phi}}(t+\theta)-v^{\xi}_{\phi}(t+\theta)}\}\},
\end{aligned}\end{equation} where $\norm{\cdot}_{L(X)}$ is the norm of linear operators in $L(X)$.

For any $\xi\in S$ and $t\in[0,1]$, one has 

\begin{equation}\label{DF-norm-difference-1}\begin{aligned}&\,\,\,\,\,\,\,\,\,\,\max\limits_{\mathop{0\leq t+\theta\leq 1,\, \theta\in[-1,0]}}\{\abs{v^{\xi}_{\tilde{\phi}}(t+\theta)-v^{\xi}_{\phi}(t+\theta)}\}\\
&=\max\limits_{\mathop{0\leq t+\theta\leq 1,\, \theta\in[-1,0]}}\{\abs{\int_0^{t+\theta}e^{-\mu(t+\theta-s)}f^{\prime}(x^{\tilde{\phi}}(s-1))v^{\xi}_{\tilde{\phi}}(s-1)ds-\int_0^{t+\theta}e^{-\mu(t+\theta-s)}f^{\prime}(x^{\phi}(s-1))v^{\xi}_{\phi}(s-1)ds}\}\\
&=\max\limits_{\mathop{0\leq t+\theta\leq 1,\, \theta\in[-1,0]}}\{\abs{\int_0^{t+\theta}e^{-\mu(t+\theta-s)}\big[f^{\prime}(x^{\tilde{\phi}}(s-1))-f^{'}(x^{\phi}(s-1))\big]\xi(s-1)ds}\}\\
&\leq \max\limits_{\mathop{0\leq t+\theta\leq 1,\, \theta\in[-1,0]}}\{\int_0^{t+\theta}\mid f^{\prime}(x^{\tilde{\phi}}(s-1))-f^{\prime}(x^{\phi}(s-1))\mid ds\}\\
&=\max\limits_{\mathop{0\leq t+\theta\leq 1,\,\theta\in[-1,0]}}\{\int_0^{t+\theta}\mid f^{\prime}(\tilde{\phi}(s-1))-f^{\prime}(\phi(s-1))\mid ds\}.
\end{aligned}\end{equation} It implies that for $(t,\tilde{\phi}), (t,\phi)\in [0,1]\times X$, 

\begin{equation}\label{DF-norm-difference}\norm{D_{\tilde{\phi}}F_t-D_{\phi}F_t}_{L(X)}\leq\max\limits_{\mathop{0\leq t+\theta\leq 1,\, \theta\in[-1,0]}}\{\int_0^{t+\theta}\mid f^{\prime}(\tilde{\phi}(s-1))-f^{\prime}(\phi(s-1))\mid ds\}.\end{equation} By the $C^1$-smoothness of $f$, 

$$D_{(\cdot)}F_{t}: X\rightarrow L(X) \,\,\text{is continuous for any}\,\,t\in[0,1].$$ By (\ref{reprentation-semiflow}), one has 

\begin{equation}\label{boundedness-F-01}\norm{F_t(\phi)}\leq \norm{\phi}+\max\limits_{\mid x\mid\leq\norm{\phi}}\{\mid f(x)\mid\} \,\,\text{for any}\,\,(t,\phi)\in[0,1]\times X.
\end{equation} By (\ref{derivative-evolution}), 

\begin{equation}\label{boundedness-DF-01}\norm{D_{\phi}F_t}_{L(X)}\leq \max\limits_{\mid x \mid\leq \norm{\phi}}\{\mid f^{\prime}(x)\mid+1\}\,\,\text{for any}\,\,(t,\phi)\in [0,1]\times X.
\end{equation} It then follows from (\ref{Chain-DFt}) and (\ref{boundedness-F-01})-(\ref{boundedness-DF-01})  that for each bounded subset $B$ of $X$ and $t\in\mathbb{R}^+$, there is a $M_{t,B}>0$ such that

\begin{equation}\label{D-opera-bound}\norm{F_{\tau}(\phi)}, \norm{D_{\phi}F_{\tau}}_{L(X)}\leq M_{t,B}\,\,\text{for any} \,\,(\tau,\phi)\in [0,t]\times B.
\end{equation} Note that for any $\tilde{\phi},\phi\in X$, 

\begin{equation}\label{deco-diff-DF}\begin{aligned}&\,\,\,\,\,\,\,\,\norm{D_{\tilde{\phi}}F_{t_1+t_2}-D_{\phi}F_{t_1+t_2}}_{L(X)}\\
&=\norm{D_{F_{t_1}(\tilde{\phi})}F_{t_2}\circ D_{\tilde{\phi}}F_{t_1}-D_{F_{t_1}(\phi)}F_{t_2}\circ D_{\phi}F_{t_1}}_{L(X)}\\
& \leq\norm{D_{F_{t_1}(\tilde{\phi})}F_{t_2}-D_{F_{t_1}(\phi)}F_{t_2}}_{L(X)}\cdot \norm{D_{\tilde{\phi}}F_{t_1}}_{L(X)}\\
&\,\,\,\,\,\,+\norm{D_{\tilde{\phi}}F_{t_1}-D_{\phi}F_{t_1}}_{L(X)}\cdot \norm{D_{F_{t_1}(\phi)}F_{t_2}}_{L(X)}.
\end{aligned}\end{equation} By mathematical induction, one has that

$$D_{(\cdot)}F_{t}: X\rightarrow L(X) \,\,\text{is continuous for any}\,\,t\in\mathbb{R}^+.$$

Therefore, we have completed the proof.
\end{proof}

\begin{rem}\label{R:Cones-imply-ES} Take any $t_0\in\mathbb{R}^+$ and compact set $\mathcal{K}\subset X$. By (\ref{D-opera-bound}), there is a $M_{t_0,\mathcal{K}}>0$ such that $\norm{F_t(\phi)},\,\norm{D_{\phi}F_t}_{L(X)}\leq M_{t_0,\mathcal{K}}$ for any $(t,\phi)\in [0,t_0]\times \mathcal{K}$.
\end{rem}

\begin{lemma}\label{Le: dissipativity} $F_t$ is point dissipative in $X$, i.e., there is a $b(\mu,f)\in (0,\infty)$ such that $\lim\sup\limits_{t\rightarrow +\infty} \norm{F_t(\phi)}\leq b(\mu,f)$ for any $\phi\in X$.
\end{lemma}
\begin{proof} It is directly implied by Lemma \ref{global-solution}.
\end{proof}

\begin{lemma}\label{Le: injectivity and compactness}
\item[(i)] $F_t$ is injective for any $t\in \mathbb{R}^+$, and it is completely continuous in $X$ for any $t\in [1,+\infty)$.
\item[(ii)] $D_{\phi}F_t$ is injective for any $(t,\phi)\in \mathbb{R}^+\times X$, and it is compact in $X$ for $(t,\phi)\in [1,+\infty)\times X$. 
\end{lemma}

\begin{proof}
{\it (i)} Since $F_0=\text{Id}$, it suffices to prove the injectivity of $F_t$ for the case $t>0$. For some $t>0$ and $\tilde{\phi},\phi\in X$ such that $\tilde{\phi}\neq \phi$, suppose $F_t(\tilde{\phi})=F_t(\phi)$ (i.e., $x_t^{\tilde{\phi}}=x_t^{\phi}$). It follows from the continuity of solutions of (\ref{E:DDE-sys-mu}) that there exists the smallest time $t_0>0$ such that $x_{t_0}^{\tilde{\phi}}=x_{t_0}^{\phi}$.  By (\ref{int-eq}), one has that for any $\epsilon\in(0,\min\{t_0,1\})$, 

$$\begin{aligned}&x^{\tilde{\phi}}(t_0)=x^{\tilde{\phi}}(t_0-\epsilon)+\int_{t_0-\epsilon}^{t_0}e^{-\mu(t_0-\tau)}f(x^{\tilde{\phi}}(\tau-1))d\tau,\\
&x^{\phi}(t_0)=x^{\phi}(t_0-\epsilon)+\int_{t_0-\epsilon}^{t_0}e^{-\mu(t_0-\tau)}f(x^{\phi}(\tau-1))d\tau.\end{aligned}$$ Then, one has 

$$0=\int_{t_0-\epsilon}^{t_0}e^{-\mu(t_0-\tau)}\big[f(x^{\tilde{\phi}}(\tau-1))- f(x^{\phi}(\tau-1))\big]d\tau$$ for any $\epsilon\in(0,\min\{t_0,1\})$. Since $f^{'}<0$, one has $x^{\tilde{\phi}}_{t_0-\epsilon}-x^{\phi}_{t_0-\epsilon}=0$ for any $\epsilon\in(0,\min\{t_0,1\})$, a contradiction. Thus, 

$$F_t\,\,\text{is injective for any}\,\,t>0.$$ 

Let $B$ be bounded in $X$ and $M$ be its upper bound. By (\ref{E:DDE-sys-mu}) and Remark \ref{R:Cones-imply-ES}, one has \begin{equation}\label{c1-norm-estimating} |(x^{\phi})^{\prime}(t)|\leq \mu M_{1,B}+\max\limits_{\abs{x}\leq M}\{|f(x)|\}\end{equation} for any $(t,\phi)\in [0,1]\times B$. Thus, all functions contained in $F_1(B)$ are uniformly equicontinuous in $X$. It then follows from Remark \ref{R:Cones-imply-ES} and Arzel$\grave{a}$-Ascoli theorem that $F_1(B)$ is precompact. Thus, $F_1$ is completely continuous. By utilizing Remark \ref{R:Cones-imply-ES} again, 

$$F_t\,\,\,\text{is completely continuous for any}\,\,\, t\geq 1.$$

{\it (ii)} It suffices to prove the case $t>0$. Let $\tilde{\xi},\xi \in X$ such that $\tilde{\xi}\neq\xi$.  Suppose that $D_{\phi}F_t\tilde{\xi}=D_{\phi}F_t\xi$ for some $t>0$. By the continuity of solutions of (\ref{E:LVE}), there is a smallest $t_0>0$ such that $D_{\phi}F_{t_0}\tilde{\xi}=D_{\phi}F_{t_0}\xi$. Furthermore, one has that 

\begin{equation}\begin{aligned}&(D_{\phi}F_{t_0}\tilde{\xi})(0)=(D_{\phi}F_{t_0-\epsilon}\tilde{\xi})(0)+\int_{t_0-\epsilon}^{t_0}e^{-\mu(t_0-s)}f^{\prime}(x^{\phi}(s-1))(D_{\phi}F_{s}\tilde{\xi})(-1)ds\\
&(D_{\phi}F_{t_0}\xi)(0)=(D_{\phi}F_{t_0-\epsilon}\xi)(0)+\int_{t_0-\epsilon}^{t_0}e^{-\mu(t_0-s)}f^{\prime}(x^{\phi}(s-1))(D_{\phi}F_{s}\xi)(-1)ds
\end{aligned}\end{equation} for any $\epsilon\in(0,\min\{t_0,1\})$. Hence, 

\begin{equation} \nonumber
0=\int_{t_0-\epsilon}^{t_0}e^{-\mu(t_0-s)}f^{\prime}(x^{\phi}(s-1))\cdot\big[(D_{\phi}F_{s}\tilde{\xi})(-1)-(D_{\phi}F_{s}\xi)(-1)\big]ds
\end{equation} for any $\epsilon\in(0,\min\{t_0,1\})$. It follows form $f^{\prime}<0$ that $D_{\phi}F_{t_0-\epsilon}\tilde{\xi}=D_{\phi}F_{t_0-\epsilon}\xi$ for any $\epsilon\in(0,\min\{t_0,1\})$, a contradiction. Thus, 

$$D_{\phi}F_t \,\,\text{is injective for any} \,\, (t,\phi)\in\mathbb{R}^+\times X.$$ By Remark \ref{R:Cones-imply-ES}, $D_{\phi}F_{t}: X\rightarrow X$ is bounded for any $(t,\phi)\in \mathbb{R}^+\times X$. By (\ref{E:LVE}), 

$$\norm{(D_{\phi} F_1\xi)^{\prime}}\leq\mu \norm{D_{\phi}F_1\xi}+\sup\limits_{\mid x \mid\leq\norm{\phi}}\{\mid f^{\prime}(x)\mid\}\cdot\norm{\xi}$$ for $(\xi,\phi)\in X\times X$. Let $B\subset X$ be bounded. For any $\phi\in X$, one has that $D_{\phi}F_{1}(B)$ is bounded and more, functions in $D_{\phi}F_{1}(B)$ are uniformly equicontinuous. By utilizing Arzel$\grave{a}$-Ascoli theorem again, one has that $ D_{\phi}F_{1}(B)$ is precompact for any $\phi\in X$ and bounded set $B\subset X$. Thus, $D_{\phi}F_{1}$ is compact for any $\phi\in X$. Together with Remark \ref{R:Cones-imply-ES}, one has 

$$D_{\phi}F_{t}\,\,\text{is compact for any}\,\, (t,\phi)\in[1,\infty)\times X.$$ 

Therefore, we have completed the proof. 
\end{proof}

\begin{lemma}\label{Le-attractor} $F_t$ has a unique connected compact global attractor $\mathcal{A}$ in the phase space $(X,\norm{\cdot})$. Consequently, one has $0\in \mathcal{A}$.
\end{lemma}

\begin{proof} By virtue of Lemma \ref{Le: dissipativity} and \ref{Le: injectivity and compactness}, $F_t$ is point dissipative and completely continuous for any $t\geq1$. It then follows from \cite[Theorem 5.2 in Chapter 4]{Hale-Lun} (see also \cite[Theorem 3.4.8]{Hale}) that $F_t$ has a connected compact global attractor $\mathcal{A}$ in $(X,\norm{\cdot})$. We emphasize that the uniquiness of $\mathcal{A}$ is implied by \cite[Theorem 3.4.2]{Hale}.
\end{proof}

For simplicity, we hereafter denote by $O^+(\phi)$ the positive semiorbit of $F_t$ originating from $\phi\in X$, $\omega(\phi)$ the $\omega$-limit set of $O^+(\phi)$, and $\mathcal{A}$ the compact global attractor of $F_t$. 

\begin{lemma}\label{flow-extension} In the phase space $X$, the following items hold:
\item[(i)] Let $\mathcal{K}\subset X$ be a compact invariant set w.r.t. $F_t$. Then, $F_t$ admits a flow extension on $\mathcal{K}$. 
\item[(ii)]For any $\phi\in X$,  $O^+(\phi)$ is precompact in $X$, and hence, $F_t$ admits a flow extension on $\omega(\phi)$.
\end{lemma}

\begin{proof} (i) By utilizing Lemma \ref{Le: injectivity and compactness}(i), $F_t$ is injetive for any $t\in \mathbb{R}^+$.  So, there exists the inverse of $F_t$ on $\mathcal{K}$ for any $t\geq 0$, denoted by $F_{-t}$. Moreover, there is a unique negative semiorbit $O^-(\phi)$ for any $\phi\in\mathcal{K}$. Recall that $F: \mathbb{R}^+\times X\mapsto X$ is continuous and possesses a unique connected compact global attractor $\mathcal{A}$ mentioned in Lemma \ref{Le-attractor}, which is maximal. Thus, $\mathcal{K}\subset \mathcal{A}$. Together with Equation (\ref{E:DDE-sys-mu}), all functions in $\mathcal{K}$ are uniformly equicontinuous. Note that for any $(t,\phi)\in\mathbb{R}^+\times \mathcal{K}$ and $(\Delta t, \tilde{\phi})\in \mathbb{R}\times \mathcal{K}$ such that $\abs{\Delta t}, \norm{\tilde{\phi}-\phi}$ small enough, one has that

$$\begin{aligned}&\norm{F_{-t-\Delta t}(\tilde{\phi})-F_{-t}(\phi)}\leq\norm{ F_{-t-\Delta t}(\tilde{\phi})-F_{-t}(\tilde{\phi})}+\norm{F_{-t}(\tilde{\phi})-F_{-t}(\phi)}\\
&\leq \norm{ F_{\abs{\Delta t}}(\psi)-\psi}+\norm{F_{-t}(\tilde{\phi})-F_{-t}(\phi)}, \,\, {\rm where}\,\, \psi=F_{\min\{-t-\Delta t,  -t\}}(\tilde{\phi})\in\mathcal{K}.\end{aligned}$$ It then suffices to prove the continuity of $F_{-t}\mid_{\mathcal{K}}$ for any fixed $t\in\mathbb{R}^+$.

By utilizing (\ref{int-eq}) and uniform equicontinuity of all funtions in $\mathcal{K}$, one can prove by contrary that for any fixed $t\in[0,1]$, $\tilde{\phi}: \mapsto  f(F_{-t}(\tilde{\phi})(\theta))$ is continuous on $\mathcal{K}$ for each $\theta\in[-1,0]$. The $C^1$-smoothness of $f$ and $f^{\prime}<0$ imply that the continuous inverse function of $f$ exists, denoted by $f^{-1}$.  Then, $\tilde{\phi}: \mapsto F_{-t}(\tilde{\phi})(\theta)$ is continuous on $\mathcal{K}$ for each $\theta\in[-1,0]$. By utilizing all $\psi\in\mathcal{K}$ being uniformly equicontinuous again, one has that $\tilde{\phi}: \mapsto F_{-t}(\tilde{\phi})$ is continuous from $\mathcal{K}$ to $\mathcal{K}$ for any $t\in[0,1]$. Thus, $F_{-t}\mid_{\mathcal{K}}$ is continuous for any $t\in[0,1]$. Furthermore, one has that $F_{-t}\mid_{\mathcal{K}}$ is continuous for any $t\in\mathbb{R}^+$. Therefore, $F_t$ admits a flow extension on $\mathcal{K}$.

\vskip 3mm
(ii) By Lemma \ref{Le: injectivity and compactness}(i), $F_t$ is completely continuous for any $t\geq 1$. Together with Lemma \ref{Le-attractor}, $O^+(\phi)$ is precompact in the phase spece $X$. Hence, $\omega(\phi)$ exists for any $\phi\in X$, and more, $\omega(\phi)$ is compact and invariant w.r.t. $F_t$. It then follows from Lemma \ref{flow-extension}(i) that $F_t$ admits a flow extension on $\omega(\phi)$.

Therefore, we have completed the proof.
\end{proof}

Let $f^{\prime}(0)=-\beta$ and hence, $\beta>0$. The linearized variable equation $(\ref{E:LVE})$ along the equilibrium $0$ is \begin{equation}\label{E:LVE-0}v^{\prime}(t)=-\mu v(t)-\beta v(t-1).\end{equation} Its characteristic equation is \begin{equation}\label{Chara-eq}-\mu=\beta e^{-\lambda}+\lambda.\end{equation} Let $\Sigma$ consist of the eigenvalues of the infinitesimal generator of $D_0F_t$ (i.e., roots of (\ref{Chara-eq})).

\begin{lemma}\label{Spectrum-D_0F}
\item[(i)] $\Sigma$ consists of complex conjugate pairs of eigenvalues in the double strips $\Sigma_k$, given by $2k\pi<|{\rm Im}(\lambda)|<2k\pi+\pi,\,\,k\in \mathbb{N}\setminus\{0\}$, and by $|{\rm Im}(\lambda)|<\pi,\,k=0$;
\item[(ii)] The total multiplicity of $\Sigma$ in $\Sigma_0$ is 2;
\item[(iii)]  $\max \{{\rm Re}(\cup_{k\in\mathbb{N}\setminus\{0\}}(\Sigma\cap \Sigma_k)) \}<\min\{{\rm Re}(\Sigma\cap \Sigma_0)\}$, where {\rm Re(I)} consists of all real parts of complex numbers in the given set ${\rm I}\subset \mathbb{C}$;
\item[(iv)]  Let $L$ be the generalized eigenspace associated with $\Sigma\cap \Sigma_0$ and $Q$ be the generalized eigenspace associated with $\cup_{k\in\mathbb{N}\setminus\{0\}} \big(\Sigma\cap\Sigma_k\big)$. Then, $X=L\oplus Q$; $L\subset X^1$ is 2 dimensional; and $L,\,Q$ are positively invariant w.r.t. $D_0F_t\,\,(t\geq 0)$.
\end{lemma}

\begin{proof} See \cite[Chapter 5]{Wal-3}. 
\end{proof}

\begin{lemma}\label{Sigma0}For $\Sigma\cap \Sigma_0$, one of the following three alternatives must hold:
\item[(i)] $\beta e^{\mu}<\frac{1}{e}:$ $\Sigma\cap \Sigma_0$ consists of two negative simple eigenvalues $u_2<u_1<0$, and a basis $\{v_1,v_2\}$ of $L$ is given by the restrictions of functions $t:\,\mapsto e^{u_2 t}$ and $t:\,\mapsto e^{u_1 t}$ on the interval $[-1,0]$.
\item[(ii)]  $\beta e^{\mu}=\frac{1}{e}:$ $\Sigma\cap\Sigma_0$ consists of a double eigenvalue $u_{00}=-\mu-1$, and a basis of $L$ is given by the restrictions of functions $t:\,\mapsto e^{u_{00}t}$ and $t:\,\mapsto t\cdot e^{u_{00}t}$ on the interval $[-1,0]$.
\item[(iii)]  $\beta e^{\mu}>\frac{1}{e}:$ $\Sigma\cap\Sigma_0$ consists of a pair of conjugate eigenvalues $\lambda=x+i\cdot y$, and $\overline{\lambda}=x-i\cdot y$ with $0<y<\pi$, and a basis of $L$ is given by $t:\,\mapsto e^{xt}\sin(yt)$ and $t:\,\mapsto e^{xt}\cos(yt)$ restricted on $[-1,0]$.
\end{lemma}

\begin{proof} See \cite[Chapter 5]{Wal-3}. 
\end{proof}

\section{Properties related to $k$-cones}

Recall that $\text{Cl}_{X}\mathcal{S}=\mathcal{S}\cup\{0\}$ and $\mathcal{S}^{1}=\mathcal{S}\cap X^1$. Thus, $\text{Cl}_{X^1}\mathcal{S}^{1}=\mathcal{S}^{1}\cup\{0\}$. 

\begin{lemma}\label{Int-bounary-S1} 

$$\begin{aligned}\text{{\rm Int}}_{X^1}\mathcal{S}^{1}=\{&\psi\in \mathcal{S}^1: \,\,\exists\,\,j\in \mathbb{N}\,\,\text{s.t.}\,\,(-1)^{j}\psi>0\,\,\text{on}\,\,[-1,0];\,\,\text{or,}\\ 
&\psi\,\, \text{possesses a unique zero}\,\,\tau\in[-1,0],\,\, \text{and}\,\,\tau\,\, \text{satisfies}\,\,\psi^{\prime}(\tau)\neq0.\}\end{aligned}$$

$$\begin{aligned}\partial_{X^1}\mathcal{S}^{1}=\big\{&\psi\in \mathcal{S}^1: \,\,\psi\,\, \text{possesses a zero}\,\,\tau\in [-1,0]\,\,\text{such that}\,\,\psi^{\prime}(\tau)=0;\\
&{\rm or\,\,else,\,there \,are \,at \,least \,\,two \,\,zeros\,such\, that \,any\,zero}\,\tau\, {\rm with}\,\psi^{\prime}(\tau)\neq 0,\\
&{\rm and \,one\, of\, zeros\, is}\,\, 0\, \,{\rm or}\,\,-1 \big\}\cup \{0\}.\end{aligned}$$ \end{lemma}

\begin{proof} We classify the points of $\mathcal{S}^{1}$ into the following three types. 

{\it Type I: $\psi\in\mathcal{S}^{1}$ such that $(-1)^j\psi>0$ on $[-1,0]$ for some $j\in\mathbb{N}$.}

Then, there is a $m>0$ such that $(-1)^j\psi>m$ on the interval $[-1,0]$. Let $\text{Int}_{X^1} B_{\frac{m}{2}}^{X^1}(\psi)=\{\xi\in X^1:\,\,\norm{\xi-\psi}_{X^1}<\frac{m}{2}\}$. Hence, $(-1)^j\xi>\frac{m}{2}$ for any $\xi\in \text{Int}_{X^1} B_{\frac{m}{2}}^{X^1}(\psi)$, and $\psi$ is an interior point of $\mathcal{S}^{1}$ in the space $X^1$. 

{\it Type II: $\psi\in\mathcal{S}^{1}$ possessing a zero $\tau\in [-1,0]$ such that $\psi^{\prime}(\tau)=0$.} 

It is clear that one can find a sesquence of points $\{\psi_n\}_{n=1}^{\infty}$ with sign change more than tiwce around $\tau$ such that $\lim\limits_{n\rightarrow \infty}\norm{\psi_n-\psi}_{X^1}=0$, and hence, $\psi\in \partial_{X^1}\mathcal{S}^{1}$. 

{\it Type III: $\psi\in\mathcal{S}^{1}$ possessing a zero $\tau\in [-1,0]$ such that $\psi^{\prime}(\tau)\neq0$, which is not of Type II.}

Since $\psi$ has at most one sign change, there is a unique zero $\tau\in [-1,0]$ with the derivative $\psi^{\prime}(\tau)\neq0$; or else, there are at least two zeros such\ that any zero $\tau$ with $\psi^{\prime}(\tau)\neq 0$ and one of zeros is 0 or -1. For the later case, it is clear that one can find two sign change functions $\psi_n$ on $[-1,0]$ such that $\lim\limits_{n\rightarrow +\infty}\norm{\psi_n-\psi}_{X^1}=0$, and hence $\psi\in\partial_{X^1}\mathcal{S}^{1}$. For the previous case, together with $\psi$ being not of Type II, there is a $j\in \mathbb{N}^+$ such that $(-1)^j\psi<0$ in $[-1,\tau)$ and $(-1)^j\psi>0$ in $(\tau,0]$. Here, we treat $[-1,\tau)=\emptyset$ if $\tau=-1$, and $(\tau,0]=\emptyset$ if $\tau=0$. Let $m=(-1)^j\psi^{\prime}(\tau)$, and clearly, $m>0$. We only discuss the subcase $\tau\in(-1,0)$ and the arguments for other subcases are similiar. Clearly, one can find $s\in (0,1)$ such that $(-1)^j\psi^{\prime}(t)>\frac{m}{2}$ for any $t\in(\tau-s,\tau+s)\subset [-1,0]$. Then, $\min\limits_{t\in[-1,\tau-s]\cup[\tau+s,0]}\{\mid\psi(t)\mid\}$ exists and is greater than 0. Let $\tilde{m}=\min\{\frac{m}{2}, \min\limits_{t\in[-1,\tau-s]\cup[\tau+s,0]}\{\mid\psi(t)\mid\}\}$ and $\text{Int}_{X^1}B^{X^1}_{\frac{\tilde{m}}{2}}(\psi)=\{\xi\in X^1:\,\,\norm{\xi-\psi}_{X^1}<\frac{\tilde{m}}{2}\}$. Then, $\text{Int}_{X^1}B^{X^1}_{\frac{\tilde{m}}{2}}(\psi)\subset \mathcal{S}^{1}$ and hence, $\psi\in \text{Int}_{X^1}\mathcal{S}^{1}$.

Clearly, $0\in \partial \mathcal{S}^{1}$, $\text{Cl}_{X^1}\mathcal{S}^1=\text{Cl}_{X}\mathcal{S}\cap X^{1}$ and $\text{Cl}_{X}\mathcal{S}=\mathcal{S}\cup \{0\}$. Therefore, we have the conclusions in this lemma.
\end{proof}

Let $L$ and $Q$ be the 2-dimensional and 2-codimensional generalized eigenspaces mentioned in Lemma \ref{Spectrum-D_0F}, which satisfies $X=L\oplus Q$. Hereafter, we denote by $\Pi_L$ the projection onto $L$ along $Q$ and then, $\Pi_Q=I-\Pi_L$ is the projection onto $Q$ along $L$. Let $Q^1=Q\cap X^1$. Recall that $L\subset X^1$. Then, $X^1=L\oplus Q^1$ and $\Pi_L\mid_{X^1}$ (resp. $\Pi_Q\mid_{X^1}$) are the projection onto $L$ along $Q^1$ (resp. onto $Q^1$ along $L$). For simplicity, we also write $\Pi_L\mid_{X^1}$ (resp. $\Pi_Q\mid_{X^1}$) as $\Pi_L$ (resp. $\Pi_Q$), when the phase space is $X^1$.

\begin{lemma}\label{S-2-cone} ${\rm Cl}_{X}\mathcal{S}$ is a complemented solid 2-cone in $X$, and ${\rm Cl}_{X^{1}}\mathcal{S}^{1}$ is a complemented 2-solid cone in $X^1$. 
\end{lemma}

\begin{proof} By the definition of $\mathcal{S}$, one can see that $\kappa \mathcal{S}= \mathcal{S}$ for any $\kappa\in \mathbb{R}\setminus\{0\}$ and $\mathcal{S}^{1}$ holds the same property. Clearly, $L\setminus\{0\}\subset \mathcal{S}^{1}$. By virtue of \cite[Lemma 5.1]{Wal-3}, one has $0\notin \Pi_L\mathcal{S}$. It implies that $Q\cap \mathcal{S}=\emptyset$ and $Q\cap {\rm Cl}_{X}\mathcal{S}=\{0\}$. Since $K^+=\{\psi\in X:\,\psi>0 \,\,\text{on}\,\,[-1,0]\}\subset {\rm Int}_X\mathcal{S}$, ${\rm Cl}_{X}\mathcal{S}$ is a complemented solid 2-cone in $X$. Note that ${\rm Cl}_{X^{1}}\mathcal{S}^{1}={\rm Cl_{X}}\mathcal{S}\cap X^1$. Thus, $Q\cap {\rm Cl}_{X^{1}}\mathcal{S}^{1}=\{0\}$. By Lemma \ref{Sigma0}, for any $\phi\in L\setminus\{0\}$, the number of potential zeros of $\phi$ is at most one, and more, $\phi^{'}(\theta)\neq 0$ if $\theta\in[-1,0]$ is the unique zero of $\phi\in L\setminus\{0\}$. It follows from Lemma \ref{Int-bounary-S1} that $\phi \in{\rm\text{Int}}_{X^1}\mathcal{S}^{1}$, and hence, $L\setminus\{0\}\subset {\rm\text{Int}}_{X^1}\mathcal{S}^{1}$. Together with $X=L\oplus Q$, $Q\cap {\rm Cl}_{X^{1}}\mathcal{S}^{1}=\{0\}$ and $X^1\subset X$, we obtain that ${\rm\text{Cl}}_{X^1}\mathcal{S}^1$ is a complemented 2-solid cone in $X^1$.
\end{proof}

\begin{lemma}\label{norm-cone-S} There is a 2-solid cone $\mathcal{C}$ satisfying $\mathcal{C} \setminus \{0\}\subset {\rm Int}_{X^1}\mathcal{S}^1$ in $X^1$, which is formulated as 

$$\mathcal{C}=\{\psi\in X^1:\,\,\norm{\Pi_Q \psi}_{X^1}\leq \kappa\norm{\Pi_L \psi}_{X^1}\}$$ for some $\kappa>0$.
\end{lemma}

\begin{proof} It follows from the arguments in the proof of Lemma \ref{S-2-cone} that $L\setminus \{0\}\subset \text{Int}_{X^1}\mathcal{S}^1$, which is 2-dimensional. Hence,  $L_1=\{\psi\in L: \norm{\psi}_{X^1}=1\}$ is compact in $X^1$. Then, there is a $\kappa>0$ such that $\cup_{\psi\in L_1}B^{X^1}_{\kappa}(\psi)\subset \text{Int}_{X^1}\mathcal{S}^1$, where $B^{X^1}_{\kappa}(\psi)=\{\tilde{\psi}\in X^1:\,\norm{\tilde{\psi}-\psi}_{X^1}\leq\kappa\} $. Hence, $\mathcal{C} \setminus \{0\}\subset \text{Int}_{X^1}\mathcal{S}^1$.
\end{proof}

\begin{lemma}\label{monotonicity} \item[(i)] $F_t (\mathcal{S})\subset \mathcal{S}$ for any $t\in \mathbb{R}^+$ and $F_t (\mathcal{S}^1)\subset \mathcal{S}^1$ for any $t\geq 1$. Moreover, ones have that $F_t({\rm Cl}_{X}\mathcal{S}\setminus\{0\})\subset{\rm Int}_{X^1}\mathcal{S}^1$ for any $t\geq 3$.

\item[(ii)] $F_t(\tilde{\phi})-F_t(\phi)\in {\rm Cl}_{X}\mathcal{S}\setminus \{0\}$ for any $t\geq 0$, and $F_t(\tilde{\phi})-F_t(\phi)\in {\rm Int}_{X^1}\mathcal{S}^{1}$ for any $t\geq 3$ whenever $\tilde{\phi}-\phi\in {\rm Cl}_{X}\mathcal{S}\setminus \{0\}$.

\item[(iii)] $F_t$ is {\rm SOP} w.r.t. ${\rm Cl}_X\mathcal{S}$.

\item[(iv)] Let $\phi\in X$. Then, $D_{\phi}F_t \,X\subset X^1$ for $t\geq 1$; and more, $D_{\phi}F_t (\mathcal{S})\subset \mathcal{S}$ for any $t\in\mathbb{R}^+$, $D_{\phi}F_t (\mathcal{S})\subset \mathcal{S}^1$ for any $t\geq 1$, and $D_{\phi}F_t (\mathcal{S})\subset {\rm Int}_{X^1}\mathcal{S}^1$ for any $t\geq 3$.
\end{lemma}

\begin{proof} {\it (i)} Let $x^{\phi}$ be the solution of Equation $(\ref{E:DDE-sys-mu})$ with the initial data $\phi\in X$. Let $y^{\tilde{\phi}}(t)=e^{\mu t}x^{\phi}(t)$ with $t\in[-1,+\infty)$, and for any $t\in\mathbb{R}^+$, $y^{\tilde{\phi}}_t(\cdot)=e^{\mu (t+\cdot)}x^{\phi}(t+\cdot)$ on the interval $[-1,0]$. Then, $y^{\tilde{\phi}}$ is the unique solution with the initial data $\tilde{\phi}=y^{\tilde{\phi}}_0$ of the following equation 

\begin{equation}\label{DDE-y}y^{'}(t)=e^{\mu t}f(e^{-\mu(t-1)}y(t-1))\,\,{\rm for\,\,any}\,\,t\geq 0.\end{equation}  Denote $g(t,\xi)=e^{\mu t}f(e^{-\mu(t-1)}\xi)$. By the assumptions on $f$ in $(\ref{E:DDE-sys-mu})$, one has that \begin{equation}\label{Cg}g(t, 0)=0\,\, \text{for any} \,\,t\in \mathbb{R}^+,\,\,\text{and}\,\,\xi \cdot g(t,\xi)<0\,\,\text{for any}\,\, (t,\xi)\in \mathbb{R}^+\times (\mathbb{R}\setminus\{0\}).\end{equation} It then follows that $y^{\tilde{\phi}}_t\in\mathcal{S}$ for any $t\geq0$ if $\phi\in\mathcal{S}$. Thus,

$$F_t (\mathcal{S})\subset \mathcal{S}\,\,\,\text{for any}\,\, t\geq 0.$$ If $\psi\in X^1$, then $x^{\psi}_t\in X^1$ for any $t\geq1$. Recall that $\mathcal{S}\cap X^1=\mathcal{S}^1$. One has that 

$$F_t(\mathcal{S}^1)\subset \mathcal{S}^1\,\,\text{for any}\,\,t\geq 1.$$

\noindent For any $\psi\in {\rm\text{Cl}}_{X}\mathcal{S}\setminus\{0\}$, suppose without loss of generality that

$$\begin{aligned} &\psi(s)> 0\,\,\text{on}\,\,[-1,0];\,\,\text{or}\\ 
& \psi(s)\geq0\,\,\text{for any}\,\, s\in[-1,z],\,{\rm and} \,\psi(s)\leq0\,\,\,\text{for any} \,\,s\in[z, 0],\,\text{where}\,\,z\in[-1,0].\end{aligned}$$ It then follows from (\ref{DDE-y}), (\ref{Cg}) and Lemma \ref{Int-bounary-S1} that $y_t^{\tilde{\psi}}\in {\rm\text{Int}}_{X^1}\mathcal{S}^1$ for any $t\geq 3$. Therefore, $F_t({\rm\text{Cl}}_{X}\mathcal{S}\setminus\{0\})\subset{\rm\text{Int}}_{X^1}\mathcal{S}^1$ for any $t\geq 3$.

{\it (ii)} Let $x^{\tilde{\phi}}$ and $x^{\phi}$ be the solution of (\ref{E:DDE-sys-mu}) with distinct initial data $\tilde{\phi}$ and $\phi\in X$ respectively. Let $y(t)=e^{\mu t}(x^{\tilde{\phi}}-x^{\phi})(t)$. Then, $y(t)$ satisfies 

\begin{equation}\label{Eq-y-m} y^{'}(t)=e^{\mu t}\big[ f(e^{-\mu(t-1)}y(t-1)+x^{\phi}(t-1))-f(x^{\phi}(t-1))\big]\,\,{\rm for\,\, any}\,\,t\geq 0.\end{equation} Let $\hat{h}(t,\xi)=e^{\mu t}\big[ f(e^{-\mu(t-1)}\xi+x^{\phi}(t-1))-f(x^{\phi}(t-1))\big]$ for any $(t,\xi)\in\mathbb{R}^+\times \mathbb{R}$. The assumptions of $f$ in (\ref{E:DDE-sys-mu}) imply that (\ref{Cg}) holds by letting $g=\hat{h}$. Then, by the same arguments in (i), we can prove that $F_t(\tilde{\phi})-F_t(\phi)\in {\rm\text{Cl}}_{X}\mathcal{S}\setminus\{0\}$ for any $t\geq0$, and $F_t(\tilde{\phi})-F_t(\phi)\in {\rm\text{Int}}_{X^1}\mathcal{S}^{1}$ for any $t\geq 3$, whenever $\tilde{\phi}, \phi \in X$ such that $\tilde{\phi}-\phi\in {\rm\text{Cl}}_{X}\mathcal{S}\setminus \{0\}$. 

{\it (iii)} By Lemma \ref{Le: embeding}, one has $F_t(X)\subset X^1$ for any $t\geq 1$. Then, for any $t\geq 1$, one can well define the map $\tilde{F}_{t}: X \mapsto X^1$ by $\tilde{F}_{t}(\phi)=F_t(\phi)$ for any $\phi\in X$. By (\ref{E:DDE-sys-mu}), one has $(\tilde{F}_{1}(\phi))^{\prime}(\theta)=-\mu\tilde{F}_{1}(\phi)(\theta)+f(\phi(\theta))$ for any $\phi\in X$ and $\theta\in[-1,0]$. Together with the continuity of $F_t\,(t\in\mathbb{R}^+)$ and the $C^1$-smoothness of $f$, one has that $\tilde{F}_{1}$ is continuous. Note that $\tilde{F}_{3}=\tilde{F}_{1}\circ\tilde{F}_{1}\circ\tilde{F}_{1}$. Then, $\tilde{F}_{3}$ is continuous. By (ii), $F_t$ is monotone w.r.t. ${\rm\text{Cl}}_{X}\mathcal{S}$, and $\tilde{F}_3(\tilde{\phi})-\tilde{F}_3(\phi)=F_3(\tilde{\phi})-F_3(\phi)\in {\rm\text{Int}}_{X^1}\mathcal{S}^{1}$ whenever $\tilde{\phi}-\phi\in {\rm\text{Cl}}_{X}\mathcal{S}\setminus \{0\}$. The continuity of $\tilde{F}_3$ implies that there are open neighborhoods of $\mathcal{U}$ of $\tilde{\phi}$ and $\mathcal{V}$ of $\phi$ in $X$ such that $\mathcal{U}\cap \mathcal{V}=\emptyset$ and $F_3(\mathcal{U})\thicksim F_3(\mathcal{V})$ w.r.t. ${\rm\text{Cl}}_{X}\mathcal{S}$. Thus, $F_t$ is SOP w.r.t. ${\rm\text{Cl}}_{X}\mathcal{S}$.

{\it (iv)} Given $\phi\in X$. Recall that $v_{\phi}^{\xi}$ is the solution of (\ref{E:LVE}) with the initial data $\xi\in X$. Together with (\ref{DFt}), one has that $D_{\phi}F_t (X)\subset X^1$ for any $t\geq 1$. Let $w_{\phi}^{\tilde{\xi}}(t)=e^{\mu t} v_{\phi}^{\xi}(t)$ with $t\in[-1,\infty)$ and $\tilde{\xi}(\theta)=e^{\mu \theta}\xi(\theta)$ with $\theta\in[-1,0]$. Then, $w_{\phi}^{\tilde{\xi}}$ is the solution with the initial data $\tilde{\xi}$ of the following equation \begin{equation}\label{VEE} w^{\prime}(t)=e^{\mu}f^{\prime}(x^{\phi}(t-1))w(t-1).\end{equation} Let $\tilde{h}(t,\eta)=e^{\mu}f^{\prime}(x^{\phi}(t-1))\eta$. Recall $f^{\prime}<0$. Then, (\ref{Cg}) holds by letting $g=\tilde{h}$. By the same arguments in (i), we can prove that $D_{\phi}F_{t} (\mathcal{S})\subset \mathcal{S}$ for any $t\geq 0$, $D_{\phi}F_{t} (\mathcal{S})\subset \mathcal{S}^{1}$ for any $t\geq 1$ and $D_{\phi}F_{t} (\mathcal{S})\subset {\rm\text{Int}}_{X^1}\mathcal{S}^{1}$ for any $t\geq 3$. Thus, we have proved (iv).
\vskip 2mm
Therefore, we have completed the proof.
\end{proof}

By Lemma \ref{Le: embeding}, $F_t(X)\subset X^1$ for any $t\geq1$. By Lemma \ref{monotonicity}(iv), $D_{\phi}F_t \,X\subset X^1$ for any $(t,\phi)\in[1,\infty)\times X$. Together with (\ref{E:DDE-sys-mu}) (resp., (\ref{E:LVE})), $C^1$-smoothness of $f$, and the continuity of $F_t$ for any $t\geq 0$ (resp., $D_{\phi}F_t$ for any $t\geq 0$ and $\phi\in X$),  the mapping $F_t\mid_{X^1}: X^1\mapsto X^1$ for any $t\geq 1$ (resp., $D_{\phi}F_t\mid_{X^1}: X^1\mapsto X^1$ for any $t\geq 1$ and $\phi\in X$) are continuous. For simplicity, we hereafter write $F_t\mid_{X^1}$ for any $t \geq1$ (resp., $D_{\phi}F_t\mid_{X^1}$ for any $t \geq1$ and $\phi\in X$) as $F_t$ (resp., $D_{\phi}F_t$). By Lemma \ref{Le: dissipativity}, $F_t$ is injective for any $t \geq0$. We hereafter write the inverse of $F_t$ on its image as $F_{-t}$ for any $t>0$. 

\begin{lemma}\label{Local-manifold-0} Let $\mathcal{C}$ be the $2$-solid cone mentioned in Lemma \ref{norm-cone-S}. Then, there exist numbers $\gamma \in\mathbb{R}, \epsilon >0$ and a $C^1$-smooth map $\overline{F}_1$ on $X^1$ with invariant manifolds $\mathcal{W}_L$, $\mathcal{W}_Q$ w.r.t. $\overline{F}_1$ such that:

\item[(i)] $\overline{F}_1=F_1$ on $B_{\epsilon}^{X^1}(0)$ and $\overline{F}_1=D_0F_1$ on $X^1\setminus B_{2\epsilon}^{X^1}(0)$, where $B_{j\epsilon}^{X^1}(0)$ is the closed ball centred at $0$ with radius $j\epsilon, j=1,2$ in $X^1$;
 
\item[(ii)] there is a $C^1$-smooth function $h_L: L\mapsto Q^1$ such that $h_{L}(0)=0$, the derivative $Dh_{L}(0)=0$ and $\mathcal{W}_L=\cap_{n\in\mathbb{N}} \overline{F}_{n}(\mathcal{C})=\{\tilde{\phi}\in X^1:\,\,\tilde{\phi}=\phi+h_{L}(\phi),\,\,\phi \in  L\}$; moreover, $\lim\limits_{n\rightarrow \infty}\norm{e^{\gamma n}\overline{F}_{-n}(\phi)}_{X^1}=0$ for any $\phi\in \mathcal{W}_L$.
\item[(iii)]  there is a $C^1$-smooth function $h_Q: Q^1\mapsto L$ such that $h_{Q}(0)=0$, the derivative $Dh_{Q}(0)=0$ and $\mathcal{W}_Q=\cap_{n\in\mathbb{N}} \overline{F}_{-n}((X^1\setminus {\rm Int}_{X^1}\mathcal{C}))=\{\tilde{\phi}\in X^1:\,\,\tilde{\phi}=\phi+h_{Q}(\phi),\,\,\phi \in Q^1\}$; moreover, $\lim\limits_{n\rightarrow +\infty}\norm{e^{-\gamma n}\overline{F}_{n}(\phi)}_{X^1}=0$ for any $\phi\in \mathcal{W}_Q$. 
\end{lemma}

\begin{proof} Recall that $v^{\xi}_{\phi}$ is the unique solution of Equation (\ref{E:LVE}) with the initial point $\xi\in X$ for any $\phi\in X$. Let $S^1=\{\psi \in X^1: \norm{\psi}_{X^1}=1\}$. By (\ref{DFt})-(\ref{Chain-DFt}), one has that for any $\tilde{\psi},\psi\in X^1$, 
\begin{equation}\begin{aligned} &\norm{D_{\tilde{\psi}}F_1-D_{\psi}F_1 }_{L(X^1)}=\sup\limits_{\xi\in S^1}\{ \norm{D_{\tilde{\psi}}F_1\xi-D_{\psi}F_1\xi}_{X^1}\}\\
&=\sup\limits_{\xi\in S^1}\{\max\limits_{\theta\in[-1,0]}\{\abs{v^{\xi}_{\tilde{\psi}}(1+\theta)-v^{\xi}_{\psi}(1+\theta)}\}+\max\limits_{\theta\in[-1,0]}\{\abs{(v^{\xi}_{\tilde{\psi}})^{\prime}(1+\theta)-(v^{\xi}_{\psi})^{\prime}(1+\theta)}\}\}\\
&\overset{(\ref{E:LVE})}{\leq} (1+\mu)\sup\limits_{\xi\in S^1}\{\max\limits_{\theta\in[-1,0]}\{\abs{v^{\xi}_{\tilde{\psi}}(1+\theta)-v^{\xi}_{\psi}(1+\theta)}\}\}+\max\limits_{\theta\in[-1,0]}\{\abs{f^{\prime}(\tilde{\psi}(\theta))-f^{\prime}(\psi(\theta))}\}\\
&\overset{(\ref{DF-norm-difference-1})}{\leq}(2+\mu)\max\limits_{\theta\in[-1,0]}\{\abs{f^{\prime}(\tilde{\psi}(\theta))-f^{\prime}(\psi(\theta))}\}.
\end{aligned}\end{equation} Clearly, $f^{\prime}$ is uniformly continuous on any nonempty bounded closed interval in $\mathbb{R}$. Then, $D_{(\cdot)}F_1: X^1\mapsto L(X^1)$ is continuous. Thus, $F_1$ is $C^1$-smooth in $X^1$. By using the technology of smooth cut-off function for the function $f$ in Equation (\ref{E:DDE-sys-mu}), we can obtain a new equation such that its solution map restricted on $X^1$ is the desired map $\overline{F}_1$.

Therefore, this lemma is implied by \cite[Theorem 5.1, Corollary 5.3]{Hir-P-S} and also \cite[Theorem 2.1-2.2]{Bat-Jo}. 
\end{proof}

\section{The open dense conjecture on enventually slowly oscillating solutions}

Recall that $\mathcal{D}_{ES}=\{\phi\in X:\,x^{\phi}\,\,\text{is an eventually slowly oscillating solution of Equation (\ref{E:DDE-sys-mu})}\}$. Denote

$$K^+=\{\psi\in X: \psi(\theta)>0\,\,\text{for all}\,\,\theta\in[-1,0]\},\quad K^-=\{\psi\in X: \psi(\theta)<0\,\,\text{for all}\,\,\theta\in[-1,0]\}.$$ Clearly, $K^+, K^-\subset \text{Int}_{X}\mathcal{S}$. Let $\phi\in X$.  Denote

$$\mathcal{M}_{\phi}=\{\tilde{\phi}\in X: \tilde{\phi}\thicksim \phi\,\,\text{and}\,\,\tilde{\phi} \neq\phi \}.$$

\subsection{Pseudo-ordered semiorbits and enventually slowly oscillating solutions}
\begin{lemma}\label{solu-orbit-1}
\item[(i)] $\phi\in \mathcal{D}_{ES}$ if and only if there exists a $t_0>0$ such that $x^{\phi}_{t_0}\in \mathcal{S}$.
\item[(ii)] If $\phi\in \mathcal{S}$, $O^+(\phi)$ is pseudo-ordered. Consequently, $O^+(\phi)$ is pseudo-ordered if $\phi\in \mathcal{D}_{ES}$.
\end{lemma}

\begin{proof} {\it (i)} It follows from Lemma \ref{Int-bounary-S1} and \ref{monotonicity}(i) that if $\phi\in \mathcal{S}$, then $x^{\phi}$ is an eventually slowly oscillating solution. Conversely, the conclusion is directly indicated by the definition of eventually slowly oscillating solutions.

{\it (ii)}  {\it Consider the case: the number of zeros in $[-1,+\infty)$ of $x^{\phi}$ is finite}. One can find a time $t_0>0$ such that the sign of $x^{\phi}(t)$, denoted by $\text{sign}(x^{\phi}(t))$, is the same as $\text{sign}(x^{\phi}(t_0))$ for any $t>t_0$. Together with (\ref{E:DDE-sys-mu}), one has that $x^{\phi}$ is strictly increasing (or decreasing) on $[t_0+1,+\infty)$. Furthermore, $F_{t_0+4}(\phi)(\theta)-F_{t_0+2}(\phi)(\theta)>0\,\, (\text{or} <0)$ for any $\theta\in[-1,0]$. Thus, $O^{+}(\phi)$ is a pseudo-ordered semiorbit.

{\it Consider the other case: the number of zeros in $[-1,+\infty)$ of $x^{\phi}$ is infinite}. By virtue of Lemma \ref{monotonicity}(i), there exists a $t_0\geq 3$ such that the following items hold: (a). $F_{t_0}(\phi), F_{t_0+1}(\phi)\in \text{Int}_{X^1}\mathcal{S}^1$; (b). $x^{\phi}(t_0)=0$ and $(x^{\phi})^{'}(t_0)\neq 0$; (c). $F_{t_0}(\phi)< 0$ on $[-1,0)$ and $F_{t_0+1}(\phi)> 0$ on $(-1,0]$. Moreover, $F_{t_0+1}(\phi)-F_{t_0}(\phi)>0$ on $[-1,0]$. It implies that $F_{t_0+1}(\phi)-F_{t_0}(\phi)\in \text{Int}_{X^1} \mathcal{S}^1$. So, $O^{+}(\phi)$ is a pseudo-ordered semiorbit.

Therefore, we have completed the proof.
\end{proof}

\begin{lemma}\label{OS-D} If $\omega(\phi)\cap \mathcal{S}\neq \emptyset$, then $\phi\in\mathcal{D}_{ES}$ and $O^+(\phi)$ is pseudo-ordered.
\end{lemma}

\begin{proof} Take $\psi\in \omega(\phi)\cap \mathcal{S}$. Note that $0$ is the unique equilibrium of $F_t$. It then follows from Lemma \ref{monotonicity}(iii) that there is an open neighborhood $\mathcal{U}$ of $\psi$ and a $t_0>0$ such that $F_t(\mathcal{U})\subset \mathcal{S}$ for any $t\geq t_0$. Hence, there is a $t_1>0$ such that $F_{t}(\phi)\in \mathcal{S}$ for any $t\geq t_1+t_0$. By Lemma \ref{solu-orbit-1}, $\phi\in\mathcal{D}_{ES}$ and $O^+(\phi)$ is pseudo-ordered.

Therefore, we have completed the proof.
\end{proof}

\begin{lemma}\label{Hans-Otto} Let $\tilde{\phi}, \phi\in X$ satisfy
\item[(i)] $O^+(\tilde{\phi}), O^+(\phi)$ are bounded by a positive number, denoted by $\max\limits_{\mathcal{A}\mid_{{\rm Cl}_{X}\mathcal{S}}}\norm{\psi}$;
\item[(ii)] there is a $t_0>0$ such that $F_{t_0}(\tilde{\phi})-F_{t_0}(\phi)\in K^+\cup K^-$.
Then there exists a constant $\delta_A>0$ such that for all $t\geq t_0+2$, one has 

\begin{equation}\label{Ratio-LQ}\begin{aligned}&\quad\norm{F_t(\tilde{\phi})-F_t(\phi)}\leq \delta_A \norm{\Pi_L\big(F_t(\tilde{\phi})-F_t(\phi)\big)} \quad\text{and hence,}\\ 
&\quad \frac{\norm{\Pi_Q\big(F_t(\tilde{\phi})-F_t(\phi)\big)}}{\norm{\Pi_L\big(F_t(\tilde{\phi})-F_t(\phi)\big)}}\leq (1+\delta_A).
\end{aligned}\end{equation} 
\end{lemma}

\begin{proof} By Lemma \ref{Le: dissipativity}, $F_t$ is always point dissipative with one of the two conditions: $\inf f>-\infty$ and $\sup f<+\infty$. Then, this lemma is implied by \cite[Proposition 6.1]{Wal-3} (See also \cite[Proposition 10.1]{Wal-2} and \cite[Proposition 7.1]{Wal-1}).
\end{proof}

\begin{lemma}\label{C0-TR} If $\tilde{\phi},\phi\in X\setminus \mathcal{D}_{ES}$ such that $\tilde{\phi}\neq \phi$ and $\omega(\tilde{\phi})=\omega(\phi)=\{0\}$, then $\tilde{\phi}\rightharpoondown \phi$.
\end{lemma}

\begin{proof} Clearly, $\tilde{\phi}-\phi\neq 0$. Prove by contrary. Suppose $\tilde{\phi}-\phi\in\mathcal{S}$. Let $x^{\tilde{\phi}}$ and $x^{\phi}$ be solutions of $(\ref{E:DDE-sys-mu})$ with initial pionts $\tilde{\phi}$ and $\phi$, respectively. Let $h_t=F_t(\tilde{\phi})-F_t(\phi)=x^{\tilde{\phi}}_t-x^{\phi}_t$ for any $t\in\mathbb{R}^+$. Clearly, $\lim\limits_{t\rightarrow +\infty}F_t(\phi)=\lim\limits_{t\rightarrow +\infty}F_t(\tilde{\phi})=\lim\limits_{t\rightarrow +\infty}h_t=0$ in $X$. Together with (\ref{E:DDE-sys-mu}), one has $\lim\limits_{t\rightarrow +\infty}F_t(\phi)=\lim\limits_{t\rightarrow +\infty}F_t(\tilde{\phi})=\lim\limits_{t\rightarrow +\infty}h_t=0$ in $X^1$ with $t\geq 1$. By Lemma \ref{monotonicity}(ii), ones have that $h_t\in\mathcal{S}$ for any $t\geq 0$ and $h_t\in \text{Int}_{X^1}\mathcal{S}^1$ for any $t\geq 3$. It then follows from Lemma \ref{Hans-Otto} that there are contants $\delta_A>0$ and $t_0>0$ such that for any $t\geq t_0$, one has 

\begin{equation}\label{h-projl-low-X}
\norm{h_{t}}\leq \delta_A\norm{\Pi_L h_{t}}\,\,\text{and}\,\,\frac{\norm{\Pi_Q h_{t}}}{\norm{\Pi_L h_{t}}}\leq1+\delta_A.
\end{equation} By Lemma \ref{Local-manifold-0}, there exist a $C^1$-smooth function $h_{Q}: Q^1 \mapsto L$ with $h_{Q}(0)=0,\,Dh_{Q}(0)=0$ and an invariant manifold $\mathcal{W}_Q$ such that 

$$\mathcal{W}_Q=\cap_{n\in\mathbb{N}} \overline{F}_{-n}((X^1\setminus\text{Int}_{X^1}\mathcal{C}))=\{\tilde{\psi}\in X^1:\,\,\tilde{\psi}=\psi+h_{Q}(\psi),\,\,\psi \in Q^1\}.$$ Recall that $\tilde{\phi},\,\phi\in  X\setminus \mathcal{D}_{ES}$. It then follows from Lemma \ref{Le: embeding} and \ref{solu-orbit-1}(i) that $F_t(\tilde{\phi}),\,F_t(\phi)\in  X^1\setminus \mathcal{S}^1$ for any sufficiently large $t\in\mathbb{R}^+$. Then, there is $t_1>t_0$ such that $F_t(\tilde{\phi}),\,F_t(\phi)\in B^{X^{1}}_{\epsilon}(0)\cap (X^1\setminus \mathcal{S}^1)$ for any $t>t_1$. By Lemma \ref{norm-cone-S}, $\mathcal{C}\setminus\{0\}\subset {\rm Int}_{X^1}\mathcal{S}^1$. Then, $F_{t}(\tilde{\phi}), F_{t}(\phi)\in \cap_{n\in\mathbb{N}} F_{-n}((X^1\setminus\mathcal{S}^1)\cap B^{X^{1}}_{\epsilon}(0)) \subset \cap_{n\in\mathbb{N}} F_{-n}((X^1\setminus\text{Int}_{X^1}\mathcal{C})\cap B^{X^{1}}_{\epsilon}(0))$ for any $t>t_1$. Note that $\overline{F}_1=F_1$ on $B^{X^1}_{\epsilon}(0)$. Hence, $F_t(\tilde{\phi}),\,F_t(\phi)\in\mathcal{W}_Q$ for any $t>t_1$. It then follows that for any $t>t_1$, 

$$ h_t=\Pi_Q h_t+h_Q(\Pi_Q\circ F_t(\tilde{\phi}))-h_Q(\Pi_Q\circ F_t(\phi))= \Pi_Q h_t+\int_0^1 Dh_Q(\Pi_Q\circ F_t(\phi)+s\Pi_Q h_t)ds\circ \Pi_Q h_t.$$ Hence, 

$$\Pi_L h_t=\int_0^1 D h_Q(\Pi_Q\circ F_t(\phi)+s\Pi_Q h_t)ds\circ \Pi_Q h_t $$ for any $t>t_1$. It then follows from $Dh_Q(0)=0$ that 

\begin{equation}\label{h-projl-low-X1}\lim\limits_{t\rightarrow \infty} \frac{\norm{\Pi_{L}h_t}_{X^1}}{\norm{\Pi_Qh_t}_{X^1}}\leq \lim\limits_{t\rightarrow \infty}\norm{\int_0^1 Dh_Q (\Pi_Q\circ F_t(\phi)+s\Pi_Q h_t)ds}_{L(X^1)}=0.
\end{equation} By Lemmas \ref{Spectrum-D_0F}-\ref{Sigma0}, there are constants $\delta_{L}^S>\delta_{L}>0$ such that for any $\xi\in L\setminus \{0\}$, 

\begin{equation}\label{low-bound-d0f}\begin{aligned} &\quad \norm{\xi^{\prime}}\leq \delta^S_L\norm{\xi} \,\,\text{and}\\
&\quad D_0F_1\xi\in L\setminus\{0\}\,\,\text{satisfying}\,\, \norm{D_0F_1\xi}\geq \delta_{L}\norm{\xi}.\end{aligned}\end{equation} By Lemma \ref{Le: smoothness}, there is a $\delta_0>0$ such that for any $\overline{\phi}\in B^{X}_{\delta_0}(0)$, 

\begin{equation}\label{pertu}
\norm{D_{\overline{\phi}}F_1-D_0F_1}_{L(X)}\leq \frac{\delta_L}{2\delta_A\norm{\Pi_L}_{L(X)}},
\end{equation} where $B^{X}_{\delta_0}(0)=\{\overline{\phi}\in X: \norm{\overline{\phi}}\leq\delta_0 \}$. By utilizing $\lim\limits_{t\rightarrow \infty} F_t(\phi)=\lim\limits_{t\rightarrow \infty} h_t=0$ in $X$ again, there is a $t_2>t_1$ such that $F_t(\phi),\,h_t\in\mathcal{B}^{X}_{\frac{\delta_0}{2}}(0)$ for any $t>t_2$. Note that for any $t\geq 1$,

$$\begin{aligned} &h_t=F_1(x^{\tilde{\phi}}_{t-1})-F_1(x^{\phi}_{t-1})=\int_0^1 D_{x^{\phi}_{t-1}+sh_{t-1}}F_1 ds h_{t-1}\\
&=D_0F_1 h_{t-1}+\int_0^1 \big( D_{x^{\phi}_{t-1}+sh_{t-1}}F_1-D_0F_1\big) ds h_{t-1}.\end{aligned}$$ Then, for any $t\geq 1$,

$$\Pi_L h_t=D_0F_1\circ \Pi_L h_{t-1}+\Pi_L \circ \int_0^1 \big( D_{x^{\phi}_{t-1}+sh_{t-1}}F_1-D_0F_1\big) ds h_{t-1}.$$ It then follows that for any $t>t_2+1$,

\begin{equation}\label{esti-t-t-1}\begin{aligned}&\norm{\Pi_L h_{t}}\overset{(\ref{h-projl-low-X})+(\ref{low-bound-d0f})+(\ref{pertu})}{\geq}\delta_L \norm{\Pi_L h_{t-1}}-\norm{\Pi_L}_{L(X)}\cdot\frac{\delta_L}{2\delta_A\norm{\Pi_L}_{L(X)}}\cdot \delta_A \norm{\Pi_L h_{t-1}} \\
&\geq \frac{\delta_L}{2}\norm{\Pi_L h_{t-1}}\overset{(\ref{h-projl-low-X})}{\geq} \frac{\delta_L}{2\delta_A}\norm{h_{t-1}}.
\end{aligned}
\end{equation} By (\ref{E:DDE-sys-mu}), one has for any $t\in\mathbb{R}^+$, 

$$h^{\prime}(t)=-\mu h(t)+\int_0^1 f^{\prime}(x^{\phi}(t-1)+sh(t-1))ds\cdot h(t-1).$$ It then follows that for any $t> t_2+1$,

\begin{equation}\label{esti-t-t-prime} \norm{h^{\prime}_t}\leq \mu \norm{h_t}+\max\limits_{\mid x\mid\leq \delta_0}\{\mid f^{\prime}(x)\mid\}\norm{h_{t-1}}\overset{(\ref{h-projl-low-X})+(\ref{esti-t-t-1})}{\leq} \delta_A\cdot \frac{\mu\delta_L+2\max\limits_{\mid x\mid\leq \delta_0}\{\mid f^{\prime}(x)\mid\}}{\delta_L}\cdot \norm{\Pi_L h_t}.
\end{equation} Note that $h_t^{\prime}=(\Pi_L h_t)^{\prime}+(\Pi_Q h_t)^{\prime}$ for any $t> t_2+1$. By (\ref{low-bound-d0f}), one has that $\norm{(\Pi_L h_t)^{\prime}}\leq \delta^S_L\norm{\Pi_L h_t}$ for any $t> t_2+1$. So, one has that for any $t>t_2+1$,

\begin{equation}\label{esti-t-t-prime-Q} 
\begin{aligned}&\norm{(\Pi_Q h_t)^{\prime}}\leq\norm{h^{\prime}_t}+\norm{(\Pi_L h_t)^{\prime}} \leq \norm{h^{\prime}_t}+\delta^S_L\norm{\Pi_L h_t}\\
&\overset{(\ref{esti-t-t-prime})}{\leq} (\delta^S_L+\delta_A\cdot \frac{\mu\delta_L+2\max\limits_{\mid x\mid\leq \delta_0}\{\mid f^{\prime}(x)\mid\}}{\delta_L})\cdot \norm{\Pi_L h_t}.
 \end{aligned}
\end{equation} For simplicity, let $c=\delta_L^S+\delta_A\cdot \frac{\mu\delta_L+2\max\limits_{\mid x\mid\leq \delta_0}\{\mid f^{\prime}(x)\mid\}}{\delta_L}$. It then follows that for any $t>t_2+1$,

$$\begin{aligned}&\frac{\norm{\Pi_{L}h_t}_{X^1}}{\norm{\Pi_Qh_t}_{X^1}}\geq \frac{\norm{\Pi_{L}h_t}}{\norm{\Pi_Q h_t}+\norm{(\Pi_Q h_t)^{\prime}}}\overset{(\ref{h-projl-low-X})+(\ref{esti-t-t-prime-Q})}{\geq} \frac{\norm{\Pi_{L}h_t}}{(1+\delta_A)\norm{\Pi_L h_t}+c\norm{\Pi_L h_t}}\\
&\quad\quad\quad\quad\quad \geq\frac{1}{1+c+\delta_A}.\end{aligned}$$ It is a contradiction to (\ref{h-projl-low-X1}). 
 
Therefore, we have completed the proof.
\end{proof}

\begin{lemma}\label{PseO-to-SO} If $O^+(\phi)$ is a pseudo-ordered semiorbit, then there is a $t_0>0$ such that $F_{t_0}(\phi)\in\mathcal{S}$, i.e., $\phi\in\mathcal{D}_{ES}$.
\end{lemma}

\begin{proof} By Lemma \ref{monotonicity}(iii), $F_t$ is SOP w.r.t. $\text{Cl}_{X}\mathcal{S}$. By Lemma \ref{flow-extension}(ii), for any $\tilde{\phi}\in X$, $O^+(\tilde{\phi})$ is precompact and $F_t$ admits a flow extension on $\omega(\tilde{\phi})$. It then follows from Lemma \ref{17-B} and $0$ being the unique equilibrium that one of the following holds:\\ 
(a) $\omega(\phi)$ is ordered; \\
(b) $\omega(\phi)$ is unordered and consists of equilibria;\\
(c) $\omega(\phi)$ possessing an ordered homoclinic property, that is,  there is an ordered and invariant subset $\tilde{B}\subset \omega(\phi)$ such that $\tilde{B}\thicksim \omega(\phi)$ and for any $\psi\in\omega(\phi)\setminus \tilde{B}$, it holds that $\alpha(\psi)\cup \omega(\psi)\subset \tilde{B}$ and $\alpha(\psi)=\{0\}$. 

Clearly, the case (b) will not occur. 
\vskip 3mm
{\it Consider that case (c) occurs}: One has that $0\in\tilde{B}\subset \omega(\phi)$ and there is a $\tilde{\phi}\in \omega(\phi)\setminus \{0\}$ such that $\tilde{\phi}\in \mathcal{S}$. It then follows from Lemma \ref{OS-D} that $\phi\in\mathcal{D}_{ES}$.

\vskip 3mm
{\it Consider that case (a) occurs}:  $\omega(\phi)$ either consists of an equilibrium (i.e., $\omega(\phi)=\{0\}$) or is a nontrivial ordered set. 

If $\omega(\phi)=\{0\}$, we claim that $\phi\in\mathcal{D}_{ES}$. Clearly, $\phi\neq 0$ and there are two distinct points $F_{s_1}(\phi)$ and $F_{s_2}(\phi)$ with $s_1,s_2\in \mathbb{R}^+$ such that $F_{s_1}(\phi)\thicksim F_{s_2}(\phi)$. It then follows from $\omega(\phi)=\{0\}$ and Lemma \ref{C0-TR} that $\phi\in\mathcal{D}_{ES}$.

If $\omega(\phi) $ is a nontrivial ordered set, we claim that $\omega(\phi)\cap\mathcal{S}\neq \emptyset$. By Lemma \ref{OS-D} again, $\phi\in\mathcal{D}_{ES}$.

Finally, we prove the claim above. Suppose that $\omega(\phi) $ is a nontrivial ordered set such that $\omega(\phi)\cap\mathcal{S}=\emptyset$. Then, $0\notin \omega(\phi)$. By Lemma \ref{S-2-cone}, $\text{Cl}_{X}\mathcal{S}$ is a complemented solid 2-cone in $X$. Note that $0$ is the unique equilibrium. It then follows from Lemma \ref{17-C}, \ref{monotonicity}(ii)-(iii) and \ref{solu-orbit-1}(i) that $\omega(\phi)\subset X\setminus \text{Cl}_{X}\mathcal{S}$ is a nontrivial periodic orbit, which is an ordered set, and its solution, denoted by $x^{\tilde{\psi}}$, is rapidly oscillating. Suppose that the mininal period of $x^{\tilde{\psi}}$ is $T_0$. By (\ref{E:DDE-sys-mu}), $x^{\tilde{\psi}}$ is $C^2$-smooth. It furthermore from the uniqueness of solutions originating from its initial point of Equation (\ref{E:LVE}) and Lemma \ref{monotonicity}(ii), (iv) that for any $t\in\mathbb{R}^+$ and $\epsilon\in (0,T_0)$, $x^{\tilde{\psi}}_{t+\epsilon}-x^{\tilde{\psi}}_{t}\in \text{Int}_{X^1}\mathcal{S}^1$ and $(x^{\tilde{\psi}})^{\prime}_{t}\in \text{Int}_{X^1}\mathcal{S}^1$. Note that for any two zeros $s,\tau$ with $s<\tau$ of $x^{\tilde{\psi}}$, there must be a zero of $(x^{\tilde{\psi}})^{\prime}_{t}$ in $[s,\tau]$. Then, for any zero $s$ of $x^{\tilde{\psi}}$, there are zeros $s_1, s_2$ of $x^{\tilde{\psi}}$ with $s<s_1<s_2$ such that $x^{\tilde{\psi}}(t)\neq 0$ on $(s,s_1)\cup(s_1,s_2)$ and $s_2-s>1$. It then follows that $x^{\tilde{\psi}}_{s_2}\in \mathcal{S}$. By virtue of Lemma \ref{solu-orbit-1}(i), $x^{\tilde{\psi}}$ is eventually slowly oscillating, a contradiction to $x^{\tilde{\psi}}$ being rapidly oscilatting. Thus, the claim has been proved.
 
\vskip 3mm
Therefore, we have completed the proof.
\end{proof}

\subsection{Proof of the main theorem}

{\bf Main Theorem:}  $\mathcal{D}_{ES}$ is open and dense in $X$.
\vskip 3mm
By the main theorem, we immediately have the following corollary, which gives an affirmative answer to the open dense conjecture on eventually slow oscillations posed by Kaplan and York.
\begin{cor} The set of initial points of an eventually slowly oscillating solution of Equation {\rm(\ref{E:DDE-sys-mu})} for $\mu=0$, is open and dense in $X$.
\end{cor}

\begin{rem}In \cite{K-Y}, Kaplan and Yorke denote 

$ C_*=\{\phi\in X: \phi\,{\rm has\,at\, most\, one\, zero;\,and\,if}\,\tau\in(-1,0)\,{\rm is\,a\,zero\,of}\,\phi,\,{\rm then}\,\phi\,{\rm must \,change\,sign\,there.}\}$. They say a solution $x$ of {\rm(\ref{E:DDE-sys-mu})} with $\mu=0$ is {\it slowly oscillating on a interval $[t_0,+\infty)$ contained in its domain} if $x_t\in C_{*}$ for any $t\geq t_0+1$. Kaplan and Yorke conjecture that the set 

$$C_{**}=\{\phi\in X: x^{\phi}_t\in C_* \,{\rm for\,all\,sufficiently\,large\,}t.\}$$ is open and dense in $X$.

Clearly, ${\rm Int}_{X^1}\mathcal{S}^1\subset C_*\subset \mathcal{S}$. It then follows from Lemma \ref{monotonicity}(i) that $\mathcal{D}_{ES}$ equals to $C_{**}$. Therefore, the conjecture of Kaplan and Yorke is for the special case $\mu=0$ of {\rm(\ref{E:DDE-sys-mu})}.
\end{rem}

Before proving the main theorem, we give a technical lemma, which is inspired by \cite[Proposition 4.3]{M-PWalther94}. 

\begin{lemma}\label{order-bet-rap} If $\tilde{\phi},\,\phi\in X\setminus \mathcal{D}_{ES}$ satisfying $\tilde{\phi}-\phi\in\mathcal{S}$, then there are $t_1,t_2>0$ such that $F_{t_1}(\tilde{\phi})\rightharpoondown F_{t_2}(\phi)$, i.e., $F_{t_1}(\tilde{\phi})-F_{t_2}(\phi)\in X\setminus {\rm Cl}_{X}\mathcal{S}$.
\end{lemma}

\begin{proof} By Lemma \ref{monotonicity}(ii) and $\text{Cl}_{X}\mathcal{S}\setminus\{0\}=\mathcal{S}$, one has $x^{\tilde{\phi}}_t-x^{\phi}_t\in \mathcal{S}$ for any $t\in\mathbb{R}^+$. Then, $x^{\tilde{\phi}}_t, x^{\phi}_t\neq 0$ for any $t\in\mathbb{R}^+$. \big(Otherwise, either $x^{\tilde{\phi}}_t=x^{\phi}_t=0$, a contradiction to $x^{\tilde{\phi}}_t-x^{\phi}_t\in\mathcal{S}$; or one of $x^{\tilde{\phi}}_t,\, x^{\phi}_t\in \mathcal{S}$, then by Lemma \ref{solu-orbit-1}, one of $\tilde{\phi},\,\phi$ is in $\mathcal{D}_{ES}$, a contradiction to $\tilde{\phi},\,\phi\in X\setminus \mathcal{D}_{ES}$.\big) By Lemma \ref{Le: injectivity and compactness}(i), $F_t$ is injective for any $t\in\mathbb{R}^+$. Together with Lemma \ref{PseO-to-SO}, one has that $F_{t_1}(\tilde{\phi})\neq F_{t_2}(\phi)$ for any $t_1,t_2\in\mathbb{R}^+$.

Now, we prove by contrary.  Under the assumptions of this lemma, suppose that $F_{t_1}(\tilde{\phi})\thicksim F_{t_2}(\phi)$ and $F_{t_1}(\tilde{\phi})\neq F_{t_2}(\phi)$ (i.e., $F_{t_1}(\tilde{\phi})-F_{t_2}(\phi)\in \mathcal{S}$) for any two distinct times $t_1, t_2\in \mathbb{R}^+$. Recall that $\phi\in X\setminus \mathcal{D}_{ES}$, $x^{\phi}_t\neq 0$ and $F_t$ is injective for any $t\in\mathbb{R}^+$. By utilizing Lemma \ref{solu-orbit-1}(i), one can find $\tau_1,\tau_2\in \mathbb{R}^+$ such that $7<\tau_1<\tau_2$, $\tau_2-\tau_1<1$, $x^{\phi}(\tau_1)=0=x^{\phi}(\tau_2)$, and $x^{\phi}(\tau)>0$ for any $\tau\in(\tau_1,\tau_2)$. Since $\tilde{\phi}\in X\setminus \mathcal{D}_{ES}$ and $x^{\tilde{\phi}}_t\neq 0$ for any $t\in\mathbb{R}^+$, by utilizing Lemma \ref{solu-orbit-1}(i) again, one can find $\tilde{\tau}_1,\tilde{\tau}_2 \in \mathbb{R}^+$ such that $\tau_2<\tilde{\tau}_1<\tau_2+1$, $\tilde{\tau}_1<\tilde{\tau}_2$, $\tilde{\tau}_2-\tilde{\tau}_1<1$, $x^{\tilde{\phi}}(\tilde{\tau}_1)=0=x^{\tilde{\phi}}(\tilde{\tau}_2)$ and $x^{\tilde{\phi}}(\tau)>0$ for any $\tau\in(\tilde{\tau}_1,\tilde{\tau}_2)$. Let $\eta=\tilde{\tau}_2-\tau_2$. Clearly, $0<\eta<2$ and $F_{\eta}(\tilde{\phi})-\phi\in \mathcal{S}$. Denoted by $d(\tau)=F_{\eta+\tau}(\tilde{\phi})(0)-F_{\tau}(\phi)(0)$ for any $\tau>0$ and $d(\theta)=F_{\eta}(\tilde{\phi})(\theta)-\phi(\theta)$ for any $\theta\in[-1,0]$. Then, $d_{\tau}=F_{\eta+\tau}(\tilde{\phi})-F_{\tau}(\phi)$ for any $\tau\in\mathbb{R}^+$. It then follows from Lemma \ref{monotonicity}(ii) that $d_{\tau}\in \text{Int}_{X^1}\mathcal{S}^1\subset \mathcal{S}$ for any $\tau\geq 3$. Together with Lemma \ref{Int-bounary-S1}, one has $d(\tau_2)=0$ with $d^{\prime}(\tau_2)\neq 0$.

\vskip 2mm
{\it Consider the case $d^{\prime}(\tau_2)>0:$} 

Then, one has that \begin{equation}\label{d-t2-so} d(\tau)>0\,\, \text{for any}\,\, \tau\in(\tau_2,\tau_2+1]\,\,\text{and}\,\, d(\tau)<0\,\, \text{for any} \,\, \tau\in[\tau_2-1,\tau_2).\end{equation}
 Recall that $\tau_1\in(\tau_2-1,\tau_2)$. Note that $0\leq x^{\tilde{\phi}}(\tau+\eta)<x^{\phi}(\tau)$ for any $\tau\in[\tilde{\tau}_1-\eta,\tau_2)$. It then follows from $x^{\phi}(\tau_1)=0$ and $x^{\phi}(\tau)>0$ for any $\tau\in(\tau_1, \tau_2)$ that $\tau_1<\tilde{\tau}_1-\eta$. Note that $(x^{\tilde{\phi}})^{\prime}(\tau_2+\eta)=(x^{\tilde{\phi}})^{\prime}(\tilde{\tau}_2)\leq 0$. By  $d^{\prime}(\tau_2)>0$, one can obtain that $(x^{\phi})^{\prime}(\tau_2)<0$ and hence, 
 
 \begin{equation} \exists\,\,\epsilon_0>0\,\,\text{such that for any}\,\,\epsilon\in(0,\,\epsilon_0),\,\,x^{\phi}(\tau_2+\epsilon)<0.\end{equation} By utilizing Lemma \ref{solu-orbit-1}(i) and $\phi\in X\setminus (\mathcal{D}_{ES}\cup \{0\})$, there is the smallest zero $\tau_3$ of $x^{\phi}$ in $(\tau_2,\tau_1+1)$. Then, $x^{\phi}(\tau)<0$ for any $\tau\in(\tau_2,\tau_3)$ and hence, $(x^{\phi})^{\prime}(\tau_3)\geq 0$. Let $\tilde{\tau}_3$ be the largest zero of $x^{\tilde{\phi}}$ in $[\tilde{\tau}_2, \tau_3+\eta]$. By (\ref{d-t2-so}), one has $x^{\tilde{\phi}}(\tau_3+\eta)>x^{\phi}(\tau_3)=0$, and then $\tau_2 \leq\tilde{\tau}_3-\eta<\tau_3$. Thus, one has 
 
\begin{equation} x^{\tilde{\phi}}(\tilde{\tau}_3)=0\,\,\text{and}\,\,(x^{\tilde{\phi}})^{\prime}(\tilde{\tau}_3)\geq 0.\end{equation}
  
Let $\delta=\tau_3-\tilde{\tau}_3+\eta$, $h(\tau)= \begin{cases} F_{\eta+\tau}(\tilde{\phi})(0)-F_{\delta+\tau}(\phi)(0)\quad\tau\geq 0\\
F_{\eta}(\tilde{\phi})(\tau)-F_{\delta}(\phi)(\tau)\quad\tau\in[-1,0]\end{cases}$ and $h_{\tau}(\theta)=h(\tau+\theta)$ with $\theta\in[-1,0]$. By Lemma \ref{monotonicity}(ii) again, one has that $h_{\tau}\in\text{Int}_{X^1}\mathcal{S}^1\subset \mathcal{S}$ for any $\tau\geq 3$. Together with Lemma \ref{Int-bounary-S1}, one has that $h(\tilde{\tau}_3-\eta)=x^{\tilde{\phi}}(\tilde{\tau}_3)-x^{\phi}(\tau_3)=0\,\,\text{with}\,\,h^{\prime}(\tilde{\tau}_3-\eta)\neq 0$. If $h^{\prime}(\tilde{\tau}_3-\eta)>0$, then $h(\tau)>0\,\,\text{for any}\,\,\tau\in(\tilde{\tau}_3-\eta, \tilde{\tau}_3-\eta+1]$ and $h(\tau)<0\,\,\text{for any}\,\,\tau\in[\tilde{\tau}_3-\eta-1, \tilde{\tau}_3-\eta)$. Moreover, $(x^{\tilde{\phi}})^{\prime}(\tilde{\tau}_3)> (x^{\phi})^{\prime}(\tau_3)\geq 0$. Recall that $(x^{\tilde{\phi}})^{\prime}(\tilde{\tau}_2)\leq 0$. Then, $\tilde{\tau}_3>\tilde{\tau}_2$ and hence, $\tau_2+\delta\in(\tau_2,\tau_3)$ and $x^{\phi}(\tau_2+\delta)<0$. Consequently, $h(\tau_2)=x^{\tilde{\phi}}(\tilde{\tau}_2)-x^{\phi}(\tau_2+\delta)=-x^{\phi}(\tau_2+\delta)>0$. Recall that $\tau_3-\tau_2\in(0,1)$ and $\tilde{\tau}_3-\eta\in (\tau_2,\tau_3)$. Then, $\tau_2\in(\tilde{\tau}_3-\eta-1, \tilde{\tau}_3-\eta)$ and hence, $h(\tau_2)<0$ a contradiction. Thus, $h^{\prime}(\tilde{\tau}_3-\eta)<0$. It then follows that \begin{equation}\label{h-t3-so} h(\tau)<0\,\,\text{for any}\,\,\tau\in(\tilde{\tau}_3-\eta, \tilde{\tau}_3-\eta+1],\,\,\text{and}\,\,h(\tau)>0\,\,\text{for any}\,\,\tau\in[\tilde{\tau}_3-\eta-1, \tilde{\tau}_3-\eta).\end{equation}  Moreover, $0\leq (x^{\tilde{\phi}})^{\prime}(\tilde{\tau}_3)< (x^{\phi})^{\prime}(\tau_3)$. Then, $0< (x^{\phi})^{\prime}(\tau_3)$. By utilizing Lemma \ref{solu-orbit-1}(i) and $\phi\in X\setminus (\mathcal{D}_{ES}\cup\{0\})$ again, there is the smallest zero $\tau_4$ of $x^{\phi}$ in $(\tau_3,\tau_2+1)$ such that $x^{\phi}(\tau)>0$ in $(\tau_3,\tau_4)$. Recall that $x^{\tilde{\phi}}(\tau)>0$ in $(\tilde{\tau}_3, \tau_3+\eta)$. By Lemma \ref{solu-orbit-1}(i) and $\tilde{\phi}\in X\setminus (\mathcal{D}_{ES}\cup\{0\})$, there is the smallest zero $\tilde{\tau}_4$ of $x^{\tilde{\phi}}$ in $(\tau_3+\eta, \tilde{\tau}_3+1)$ such that $x^{\tilde{\phi}}(\tau)>0$ in $(\tilde{\tau}_3, \tilde{\tau}_4)$. Note that $\tilde{\tau}_3-\eta=\tau_3-\delta$ and $\tau_3\in(\tau_2,\tau_1+1)$. Then, one has that $0<\delta<\tilde{\tau}_4-\tilde{\tau}_3$. Note that $(\tilde{\tau}_3-\eta, \,\tilde{\tau}_4-\eta)\cup (\tau_3-\delta,\tau_4-\delta)\subset (\tilde{\tau}_3-\eta, \tilde{\tau}_3-\eta+1]$. By (\ref{h-t3-so}), one has $0\leq x^{\tilde{\phi}}(\tilde{\tau}_3+\epsilon)<x^{\phi}(\tau_3+\epsilon)$ for any $\epsilon\in [0,\tilde{\tau}_4-\tilde{\tau}_3]$. Consequently, $\tilde{\tau}_4-\tilde{\tau}_3<\tau_4-\tau_3$, $(\tilde{\tau}_3-\eta, \,\tilde{\tau}_4-\eta)\subset (\tau_3-\delta,\tau_4-\delta)\subset (\tau_2, \tau_2+1] \cap(\tilde{\tau}_3-\eta, \tilde{\tau}_3-\eta+1]$, and $\tilde{\tau}_4-\eta\in (\tau_3,\tau_4)$. Then, $x^{\phi}(\tilde{\tau}_4-\eta)>0$. It furthermore follows from (\ref{d-t2-so}) that $0=x^{\tilde{\phi}}(\tilde{\tau}_4)>x^{\phi}(\tilde{\tau}_4-\eta)$, a contradiction.\\
 
 {\it Consider the case $d^{\prime}(\tau_2)<0:$}
 
Let $\tilde{d}(\tau)=x^{\phi}(\tau)-x^{\tilde{\phi}}(\tau+\eta)$ for $\tau\in[-1,+\infty)$. Then, $\tilde{d}^{\prime}(\tau_2)>0$. Note that the shape of the gragh of the function $\tilde{d}$ will not be changed no matter  $\tau_2$ is greater or smaller than $\tilde{\tau}_2$. By the same arguments for the case $d^{\prime}(\tau_2)>0$, we can deduce the contray, and then, prove that there exist $t_1,t_2>0$ such that $F_{t_1}(\tilde{\phi})-F_{t_2}(\phi)\in X\setminus \text{Cl}_{X}\mathcal{S}$.
 
Therefore, we have completed the proof.
\end{proof}

\vskip 5mm
{\it Proof of the main theorem}:

Firstly, we prove that $\mathcal{D}_{ES}$ is open, i.e., $\mathcal{D}_{ES}=\text{Int}_{X}\mathcal{D}_{ES}$. Given a $\phi\in \mathcal{D}_{ES}$. By Lemma \ref{solu-orbit-1}(i), there is a $t_0>0$ such that $F_{t_0}(\phi)\in \mathcal{S}$. Recall that $0$ is an equilibrium of $F_t$. It then follows from Lemma \ref{monotonicity}(iii) and the continuity of $F$ that such that there is an open neighborhood $\mathcal{U}_{\phi}$ of $\phi$ and a $t_{\phi}>0$ such that $F_{t_0+t_{\phi}}(\mathcal{U}_{\phi})\subset \mathcal{S}$. By Lemma \ref{solu-orbit-1}(i) again, one has $\mathcal{U}_{\phi}\subset \mathcal{D}_{ES}$. It implies that $\mathcal{D}_{ES}$ is open. 

Now, we prove that $\mathcal{D}_{ES}$ is dense. By Lemma \ref{solu-orbit-1}(i), $\mathcal{S}\subset \mathcal{D}_{ES}$. It then follows from $0\in {\rm Cl}_{X}\mathcal{S}$ that $0\in {\rm Cl}_{X}\mathcal{D}_{ES}$. So, we only need to consider the point $\phi$ in $X\setminus(\{0\}\cup \mathcal{D}_{ES})$. By virtue of Lemma \ref{OS-D}, one has that $\omega(\phi)\subset X\setminus \mathcal{S}$. Together with Lemma \ref{17-D} and \ref{PseO-to-SO}, $\omega(\phi)$ is either (a) an equilibrium (i.e., $\omega(\phi)=\{0\}$), or (b) an unordered set. 

\noindent {\it Consider case {\rm(}a{\rm)}}: It then follows from Lemma \ref{order-rela-betw-ome} and \ref{C0-TR} that $\omega(\tilde{\phi})\cap \mathcal{S}\neq \emptyset$ for any $\tilde{\phi}\in\mathcal{M}_{\phi}$. Together with Lemma \ref{OS-D}, one has $\tilde{\phi}\in \mathcal{D}_{ES}$ for any $\tilde{\phi}\in\mathcal{M}_{\phi}$. Clearly, $K^+\cup K^-\subset {\rm Int}_{X}\mathcal{S}$ and hence, $\phi+(K^+ \cup K^-)\subset \mathcal{M}_{\phi}$. So, $\phi\in{\rm Cl}_{X}\mathcal{D}_{ES}$.

\noindent {\it Consider case {\rm(}b{\rm)}}: It then follows from Lemma \ref{order-rela-betw-ome} that $\omega(\tilde{\phi})\cap \omega(\phi)=\emptyset$ and $\omega(\tilde{\phi})\thicksim \omega(\phi)$ for any $\tilde{\phi}\in \mathcal{M}_{\phi}$. Clearly, $0\notin \omega(\tilde{\phi})$ for any $\tilde{\phi}\in\mathcal{M}_{\phi}$ because of $\omega(\phi)\subset X\setminus \mathcal{S}$. If $0\in \omega(\phi)$, then $\omega(\tilde{\phi})\cap\mathcal{S}\neq \emptyset$, and by Lemma \ref{OS-D} again, $\tilde{\phi}\in\mathcal{D}_{ES}$ for any $\tilde{\phi}\in\mathcal{M}_{\phi}$. If $0\notin \omega(\phi)$, then $\tilde{\phi}\in\mathcal{D}_{ES}$ for any $\tilde{\phi}\in \mathcal{M}_{\phi}$. (Otherwise, $\tilde{\phi}\notin\mathcal{D}_{ES}$. By virtue of Lemma \ref{OS-D} again, one has that $\omega(\tilde{\phi})\subset X\setminus {\rm Cl}_{X}\mathcal{S}$. It then follows from Lemma \ref{solu-orbit-1} and the uniquiness of equilibirium of $F_t$ that $\psi,\tilde{\psi}\in X\setminus (\{0\}\cup \mathcal{D}_{ES})$ for any $\psi\in\omega(\phi)$ and $\tilde{\psi}\in\omega(\tilde{\phi})$. Together with Lemma \ref{order-bet-rap}, there are points $\psi\in\omega(\phi)$ and $\tilde{\psi}\in\omega(\tilde{\phi})$ such that $\psi\rightharpoondown \tilde{\psi}$, a contradiction to $\omega(\tilde{\phi})\thicksim \omega(\phi)$.) Thus, $\tilde{\phi}\in\mathcal{D}_{ES}$ for any $\tilde{\phi}\in\mathcal{M}_{\phi}$. By repeating the arguments in Case (a), one has $\phi\in{\rm Cl}_{X}\mathcal{D}_{ES}$.

\noindent Together with the conclusions of case (a) and case (b), one has that $\mathcal{D}_{ES}$ is dense.
\vskip 5mm
Therefore, we have completed the proof. $\quad\quad\quad\quad\quad\quad\quad\quad\quad\quad\quad\quad\quad\quad\quad\quad\quad\quad\quad\quad\quad\quad \square$

\vskip 5mm
\noindent {\bf Acknowledgments}   The author would like to appreciate that Prof. Yi Wang and Prof. Jianhong Wu propose the author to do the research on the open dense conjecture on eventually slow oscillations of the differential equation with delayed negative feedback.

\end{document}